\documentclass[11pt,a4paper]{article}

\usepackage[utf8]{inputenc}
\usepackage[T1]{fontenc}
\usepackage[english]{babel}
\usepackage{amsmath}
\usepackage{amsfonts}
\usepackage{amsthm}
\usepackage{amssymb}
\usepackage{makeidx}
\usepackage{graphicx}
\usepackage{lmodern}
\usepackage[left=2cm,right=2cm,top=2cm,bottom=2cm]{geometry}

\usepackage{color}
\usepackage{comment}
\usepackage[dvipsnames]{xcolor}
\usepackage{graphicx}
\usepackage{framed}
\usepackage{tikz}
\usepackage{calrsfs}
\DeclareMathAlphabet{\pazocal}{OMS}{zplm}{m}{n}

\usepackage{fourier-orns}
\usepackage{mathtools}
\usepackage{relsize}
\usepackage{dsfont}
\usepackage{comment}
\usepackage{hyperref}
\newcommand{\B}{\mathbb{B}}  
\newcommand{\R}{\mathbb{R}} 

\newcommand{\N}{\mathbb{N}}

\newcommand{\F}{\pazocal{F}}
\newcommand{\G}{\pazocal{G}}
\newcommand{\U}{\pazocal{U}}
\newcommand{\K}{\pazocal{K}}

\newcommand{\M}{\pazocal{M}}
\newcommand{\Qpazo}{\pazocal{Q}}
\newcommand{\Hpazo}{\pazocal{H}}
\newcommand{\Npazo}{\pazocal{N}}

\newcommand{\Ppazo}{\pazocal{P}}

\newcommand{\Opazo}{\pazocal{O}}

\newcommand{\Lcal}{\mathcal{L}}
\newcommand{\Vcal}{\mathcal{V}}
\newcommand{\Pcal}{\mathcal{P}}

\newcommand{\Scal}{\mathcal{S}}

\newcommand{\Id}{\textnormal{Id}}

\newcommand{\Tan}{\textnormal{Tan}}

\newcommand{\supp}{\textnormal{supp}}

\newcommand{\Lip}{\textnormal{Lip}}
\newcommand{\AC}{\textnormal{AC}}
\newcommand{\Graph}{\textnormal{Graph}}

\newcommand{\Div}{\textnormal{div}}

\newcommand{\dist}{\textnormal{dist}}
\newcommand{\textbn}[1]{\textnormal{\textbf{#1}}}

\newcommand{\co}{\overline{\textnormal{co}} \,}

\newcommand{\xb}{\boldsymbol{x}}
\newcommand{\yb}{\boldsymbol{y}}

\newcommand{\vb}{\boldsymbol{v}}

\newcommand{\Bgamma}{\boldsymbol{\gamma}}

\newcommand{\Bnu}{\boldsymbol{\nu}}

\newcommand{\Beta}{\boldsymbol{\eta}}
\newcommand{\Bmu}{\boldsymbol{\mu}}

\newcommand{\INTDom}[3]{\int_{#2} #1 \textnormal{d} #3}
\newcommand{\INTSeg}[4]{\int_{#3}^{#4} #1 \textnormal{d} #2}
\newcommand{\NormL}[3]{\parallel \hspace{-0.1cm} #1 \hspace{-0.1cm} \parallel _ {L^{#2}(#3)}}
\newcommand{\NormC}[3]{\left\| #1  \right\| _ {C^{#2}(#3)}}
\newcommand{\Norm}[1]{\parallel \hspace{-0.1cm} #1 \hspace{-0.1cm} \parallel}

\newtheorem{rmk}{Remark}
\newtheorem{lem}{Lemma}
\newtheorem{Def}{Definition}
\newtheorem{thm}{Theorem}
\newtheorem{prop}{Proposition}

\newtheorem{cor}{Corollary}

\newtheorem*{hyp}{Hypotheses}
\newtheorem*{hypsing}{Hypothesis}

\title{Differential Inclusions in Wasserstein Spaces: \\ the Cauchy-Lipschitz Framework}
\author{Benoît Bonnet\footnote{CNRS,  IMJ-PRG,  UMR  7586,  Sorbonne  Université, 4  place  Jussieu,  75252  Paris,  France. \hfill \hspace{3.5cm} \textit{E-mail}: \texttt{benoit.bonnet@imj-prg.fr} (Corresponding author)} , Hélène Frankowska\footnote{CNRS,  IMJ-PRG,  UMR  7586,  Sorbonne  Université,  4  place  Jussieu,  75252  Paris,  France. \hfill \hspace{3.5cm} \textit{E-mail}: \texttt{helene.frankowska@imj-prg.fr}}}

\begin{document}

\maketitle

\begin{abstract}
In this article, we propose a general framework for the study of differential inclusions in the Wasserstein space of probability measures. Based on earlier geometric insights on the structure of continuity equations, we define solutions of differential inclusions as absolutely continuous curves whose driving velocity fields are measurable selections of multifunction taking their values in the space of vector fields. In this general setting, we prove three of the founding results of the theory of differential inclusions: Filippov's theorem, the Relaxation theorem, and the compactness of the solution sets. These contributions -- which are based on novel estimates on solutions of continuity equations -- are then applied to derive a new existence result for fully non-linear mean-field optimal control problems with closed-loop controls. 
\end{abstract}

{\footnotesize

\vspace{0.25cm}

\textbf{MSC2010 Subject Classification : } 28B20, 34A60, 34G20, 49J21, 49J45
}

\section{Introduction}

During the past decade, the study of large dynamical systems appearing in the modelling of social dynamics and network analysis has taken an increasingly important place in several mathematical communities. \textit{Multi-agent systems} are ubiquitous in a wide number of applications, ranging from the understanding of patterns formation in several branches of the animal kingdom \cite{ballerini,Bertozzi2004,CS2} to the analysis of pedestrian dynamics \cite{Albi2014,CPT}, ensembles of autonomous vehicles \cite{BeardRen,Bullo2009}, and opinion formation on networks \cite{Bellomo2013,HK,MPPD}. They are also at the core of active academic fields such as \textit{mean-field games}, a theory formalised simultaneously in \cite{Huang2006} and \cite{Lasry2007} which now constitutes a prominent topic in applied mathematics. 

At the microscopic level, multi-agent systems are commonly modelled by a family of $N \geq 1$ point-trajectories $\xb_N(\cdot) := (x_1(\cdot),\dots,x_N(\cdot))$ in a given state-space (e.g. $\R^d$ or a smooth manifold), whose evolution is described by a system of coupled ordinary differential equations of the form
\begin{equation}
\label{eq:Intro_ODE}
\dot x_i(t) = \vb_N (t,\xb_N(t),x_i(t)).
\end{equation}
In this context, the velocity field $(t,x) \mapsto \vb_N (t,\xb_N(t),x )$ stirring each individual agent $x_i(\cdot)$ is \textit{non-local}, in the sense that it depends on the total state $\xb_N(t)$ of the system at each time. These types of dynamics frequently appear in the form of discrete convolutions, which are used to represent sums of bipartite interactions in the system (see e.g. \cite{CPT,CS2,HK}). It should also be noted that in several engineering-oriented communities, the analysis of multi-agent systems is carried out in a graph-theoretic framework (see e.g. \cite{Bullo2009,Egerstedt2010} and references therein). 

The investigation of multi-agent systems from a control-theoretic perspective is fairly challenging for several reasons. Firstly, the high dimensionality of the problems often prevents the application of classical finite-dimensional control or optimisation methods. Secondly, designing individual control laws at the \textit{microscopic} scale may not be relevant from a conceptual standpoint compared to the implementation of a \textit{macroscopic} signal used to pilot the system as a whole. For these reasons, control problems for multi-agents systems of the form \eqref{eq:Intro_ODE} are often studied in the so-called \textit{mean-field approximation framework}. In this setting, the collection of individual agents $(x_1(\cdot),\dots,x_N(\cdot))$ is replaced by a density $\mu(\cdot)$, whose evolution is described by a \textit{non-local continuity equation} of the form
\begin{equation}
\label{eq:Intro_CE}
\partial_t \mu(t) + \Div \big( v (t,\mu(t))\mu(t) \big) = 0.
\end{equation} 
Here, the velocity field $(t,x) \mapsto v(t,\mu(t))(x)$ is the mean-field pendant of $(t,x) \mapsto \vb_N(t,\xb_N(t),x)$, and depends on the whole density $\mu(t)$ at each time. Equations of the form \eqref{eq:Intro_CE} arise very naturally when studying mean-field limits of deterministic particle systems, see e.g. the pioneering work \cite{vlasov}. 

During the last few years, an important research effort at the interface between control theory and calculus of variations has been directed towards control problems formulated on continuity equations of the form \eqref{eq:Intro_CE}. While a few results have been dealing with controllability properties of continuity equations \cite{Duprez2019,Duprez2020}, the major part of the literature has been devoted to the study of optimal control problems, with contributions ranging from existence results \cite{LipReg,FLOS,FPR,MFOC} and necessary optimality conditions  \cite{achdou2,MFPMP,PMPWassConst,PMPWass,Burger2019,Cavagnari2018,Cavagnari2020,Pogodaev2016}
 to numerical methods \cite{Burger2020,Pogodaev2017}. All these findings have hugely benefited from theoretical progresses made in the theory of \textit{optimal transport}, for which we refer largely to the reference monographs \cite{AGS,OTAM,villani1}. 

It is now well-understood that continuity equations play a key role in the geometric study of the so-called \textit{Wasserstein spaces} $(\Pcal_p(\R^d),W_p)$ of optimal transport (see Definition \ref{def:WassSpace} below). It was noticed as early as  \cite{Benamou2000} that they are involved in the dynamical formulation of the optimal transport problem, in which one aims at computing Wasserstein geodesics by searching for curves of minimal length joining two prescribed measures. The fact that these equations represent the ``good'' class of intrinsic dynamics to work with in Wasserstein spaces was further confirmed in \cite[Chapter 8]{AGS}, where it is proven that absolutely continuous curves of measures coincide with the solutions of \eqref{eq:Intro_CE} driven by integrable velocity fields $(t,x)  \mapsto v(t,x)$ (which are independent from $\mu(\cdot)$). In the particular case where $p=2$, the approach allowing to prove this important result also provides an explicit construction of the so-called \textit{analytical tangent space} $\Tan_{\mu} \Pcal_2(\R^d)$ to the \textit{manifold} of measures $(\Pcal_2(\R^d),W_2)$ (see \cite[Section 8.4]{AGS}). This far-reaching characterisation contributed to building a solid mathematical basis for the pseudo-Riemannian  structure of $(\Pcal_2(\R^d),W_2)$, which was first explored in \cite{Otto2001}.

In addition to these geometric considerations, an extensive literature has been focusing on the well-posedness theory for continuity equations, predominantly for velocity fields $v(\cdot,\cdot)$ that do not depend on the density $\mu(\cdot)$. The classical Cauchy-Lipschitz framework (which we further detail in Section \ref{subsection:ContinuityEquations} below) was first extended to Sobolev vector fields in \cite{DipernaLions}, and later to $BV$ vector fields in \cite{AmbrosioPDE}. Several other relevant classes of vector fields have been considered since then, such as velocity fields of bounded deformation or with a Hamiltonian structure (see e.g. the survey \cite{Ambrosio2011}). More recently, a comprehensive study of ``local'' solutions in the spirit of the classical Carathéodory theory for ODEs has been presented in \cite{Ambrosio2015} for Sobolev and $BV$ vector fields. We also mention the very recent preprint \cite{Karimghasemi2020} in which flow solutions inspired by \cite{Bianchini2020} are built for continuity equations with very rough velocities by adequate combinations of Filippov regularisations and measurable selection principles. Concerning non-local continuity equations of the form \eqref{eq:Intro_CE}, some well-posedness results were first derived in \cite{AmbrosioGangbo} for Hamiltonian flows in $(\Pcal_2(\R^d),W_2)$, and a first Cauchy-Lipschitz theory was subsequently elaborated in \cite{Pedestrian}. Besides, a new formalism of \textit{measure-driven differential equations} in the spirit of the theory of Young measures (see e.g. the seminal contribution \cite{Bernard2008}) was proposed in \cite{Piccoli2019} for continuity equations.

\bigskip

In this paper, we introduce differential inclusions in Wasserstein spaces -- which are  set-valued  extensions of the  \textit{non-local continuity equation} \eqref{eq:Intro_CE} -- in the Cauchy-Lipschitz framework. Differential inclusions are an active area of research in the setting of finite and infinite dimensional vector spaces. Indeed, generalised multivalued ordinary differential equations of the form 
\begin{equation*}
\dot \sigma (t) \in F(t, \sigma(t)),
\end{equation*}
where $F:[0,T]\times \R^d \rightrightarrows \R^d $ is a \textit{set-valued map} (see Section \ref{subsection:SetValued} below) appeared in the literature as early as 1936. Since the beginning of the 60's, mathematical tools developed for differential inclusions have been successfully exploited in several branches of control theory, as it was observed that under mild assumptions, control systems could be seen as particular cases of differential inclusions. Starting from there, existence results for optimal controls were deduced from the compactness theorems on the sets of trajectories of differential inclusions, while the Relaxation theorem allowed to describe the closure of trajectories of a general non-linear control system as the set of all solutions of the differential inclusion whose right-hand side is the convexified control system.  Let us also stress that if the right-hand side $F(\cdot,\cdot)$ has closed convex values, then parametrisation theorems allow to rewrite the corresponding differential inclusion as a control system. We refer to \cite{Aubin1984,Aubin1990} for these results and for historical and bibliographical comments on the theory of differential inclusions. 

Similar studies were also performed for evolution inclusions in general Banach spaces, which led to the derivation of necessary optimality conditions in the form of an infinite dimensional maximum principle for optimal control problems involving pointwise state constraints, see \cite{FrankowskaMM2018}. However, the context of Wasserstein spaces considered in this paper is substantially different. On the one hand, only a purely metric structure is available in this setting, and on the other hand the corresponding class of dynamics represented by non-local continuity equations produces highly non-linear semigroups which are much more delicate to handle. Besides, even though the space of probability measures can be seen as a subset of the Banach space of Radon measures, the induced metric and the corresponding duality are generically not explicit enough to formulate many of the important results of control theory.
    
\bigskip    
    
The first step towards the definition of differential inclusions for continuity equations is to identify the object which plays the role of a multi-valued velocity field in this context. The answer to this question is provided by the Riemannian analogy that we have sketched previously, and which can be used to give another meaning to continuity equations in terms of \textit{ordinary differential equations} (see e.g. \cite[Chapter 3]{Gangbo2011}). Given an arbitrary non-local velocity field
\begin{equation}
\label{eq:NonLocal_VF}
(t,\mu) \in [0,T] \times \Pcal_2(\R^d) \mapsto v(t,\mu) \in L^2(\R^d,\R^d;\mu),
\end{equation}
the continuity equation \eqref{eq:Intro_CE} can be heuristically rewritten as 
\begin{equation}
\label{eq:ContinuityEquationsRHS}
\dot \mu(t) = - \Div \big( v(t,\mu(t)) \mu(t) \big),
\end{equation}
and be seen as an ODE in the flat distribution space $(C^{\infty}_c(\R^d))'$, whose velocity is the right-hand side of \eqref{eq:ContinuityEquationsRHS}. However, it can also be interpreted as a differential equation formulated in the abstract manifold of measures $(\Pcal_2(\R^d),W_2)$. In this analogy, the object which plays the role of the velocity of the curve $\mu(\cdot)$ is the application $t \in [0,T] \mapsto \pi_{\mu(t)}(v(t,\mu(t))$, where $\pi_{\mu} : L^2(\R^d,\R^d;\mu) \rightarrow \Tan_{\mu} \Pcal_2(\R^d)$ denotes the orthogonal projection onto the analytical tangent space. This interpretation -- along with several other facts listed hereinabove -- suggests that in order to turn the ``differential equation'' \eqref{eq:ContinuityEquationsRHS} into a ``differential inclusion'', the object which needs to become set-valued is the mapping \eqref{eq:NonLocal_VF}.    
    
In the present work, we therefore propose a \textit{functional approach} to Cauchy-Lipschitz differential inclusions in  Wasserstein spaces. Namely, we consider set-valued maps $V : [0,T] \times \Pcal_c(\R^d) \rightrightarrows C^0(\R^d,\R^d)$ with values in a subset of locally Lipschitz mappings from $\R^d$ into itself, and say that a curve of measures $\mu(\cdot)$ is a solution of the \textit{differential inclusion}
\begin{equation*}
\partial_t \mu(t) \in -\Div \Big( V(t,\mu(t)) \mu(t) \Big), 
\end{equation*}
if there exists a \textit{measurable selection} $t \in [0,T] \mapsto v(t) \in V(t,\mu(t))$ (see Definition \ref{def:MeasMultifunction} below) such that 
\begin{equation}
\label{eq:ContinuityIntroLocal}
\partial_t \mu(t) + \Div \big( v(t) \mu(t) \big) = 0.
\end{equation}
This setting is very convenient for investigating solutions of \eqref{eq:Intro_CE} with control-dependent velocities, whose characteristics are described by control systems of the form
\begin{equation*}
\dot \sigma (t) = v(t,\mu(t),u(t))(\sigma(t)), 
\end{equation*}
where $u : [0,T] \rightarrow U$ is a Lebesgue-measurable control, $U$ is a subset of a metric space and $v : [0,T] \times \Pcal_c(\R^d) \times U \rightarrow C^0(\R^d,\R^d)$ is a control-dependent velocity. Indeed, setting 
\begin{equation*}
V(t,\mu)  ~:=~ \Big\{ v(t,\mu, u) \in C^0(\R^d,\R^d) ~\text{s.t.}~ u \in U \Big\},
\end{equation*}
and assuming that $x \in \R^d \mapsto v(t,\mu,u)(x) \in \R^d$ is locally Lipschitz with a constant independent from $(\mu,u)$, the set $V(t,\mu)$ consists of locally Lipschitz mappings. Furthermore,  if $t \in [0,T] \mapsto \vb(t) \in V(t,\mu(t))$ is Lebesgue-measurable, then by classical measurable selection theorems, there exists a control $u:[0,T] \rightarrow U$ such that $\vb(t) := v(t,\mu(t), u(t))$ for $\Lcal^1$-almost every $t \in [0,T]$. In the other words, given a measurable selection $t \in [0,T] \mapsto \vb(t) \in V(t,\mu(t))$, we can always associate to it a control $u(\cdot)$, which is then used to write the characteristic system. The functional approach to differential inclusions in Wasserstein spaces that we propose here can be applied in particular to closed-loop controls, i.e. when $U$ is a set of locally Lipschitz functions from $\R^d$ into a metric space, see Section \ref{section:Existence} for more details.

Recently, differential inclusions in Wasserstein spaces have attracted the attention of several researchers. For instance in \cite{PiccoliCDC2018}, a notion of weak solutions is introduced  for the inclusion $\dot{\mu}(t) \in \mathbb{V}(\mu(t))$ where $\mathbb{V} : \Pcal_c(\R^d) \rightrightarrows \Pcal_c(T \R^d)$ is a set-valued map with values in the space of probability measures over the tangent bundle $T \R^d \simeq\R^d \times \R^d$. The existence of weak solutions is then investigated by following ideas akin to \cite{Piccoli2019} under a convexity-type requirement on the right-hand side of the dynamics, which is a delicate notion to handle in a metric space. In the present manuscript, we consider the differential inclusion $\partial_t \mu(t) \in - \Div(V(\mu(t))\mu(t))$, which is the direct generalisation of the continuity equation \eqref{eq:Intro_CE}. Furthermore in our case, the admissible velocities $V(\mu)$ are subsets of the vector space $C^0(\R^d,\R^d)$, which automatically lifts convexity-related issues.

We would also like to stress the difference between the theory developed in the present paper and the approach followed e.g. in \cite{Cavagnari2018,CavagnariMP2018,Jimenez2020}, where the authors investigated the existence of optimal solutions and Hamilton-Jacobi-Bellman equations associated to optimal control problems in Wasserstein spaces. In these control problems, the minimisation is taken over the set of curves solving continuity equations of the form \eqref{eq:Intro_CE}, where $(t,x) \in [0,T] \times \R^d \to v(t,x) \in \R^d$ is a Borel mapping satisfying the non-holonomic constraint  $v(t,x) \in F(\mu(t),x)$. Here $F : \Pcal_2(\R^d) \times \R^d \rightrightarrows \R^d $ is a given  set-valued map  with convex compact images. In addition to the conceptual difference between the two approaches -- which is partly due to the fact that the set-valued maps do not take values in the same spaces --, the set of admissible trajectories studied in these articles appears to be larger (in general) than the one considered in the present paper, since the admissible controls obtained by applying measurable selection principles depend on each of the characteristics of the inclusion. This particular point is discussed in greater details in Remark \ref{rmk:AdmissibleTraj} below. Moreover, the functional approach to differential inclusions developed in this paper is very close in spirit to the formulation of gradient flows \cite[Chapter 11]{AGS} and Hamiltonian flows \cite{AmbrosioGangbo} in Wasserstein spaces, which correspond to special (but not less general) cases in which the sets of admissible velocities $V(\mu) \subset L^2(\R^d,\R^d;\mu)$ are defined using the Wasserstein subdifferentials at $\mu \in \Pcal_2(\R^d)$ of a given functional.

After introducing precisely our notion of solution to differential inclusions in Section \ref{subsection:DefInc}, we will prove the generalisations to the Wasserstein spaces of three cornerstones of the theory of set-valued dynamical systems: \textit{Filippov's theorem}, the \textit{Relaxation theorem}, and the \textit{compactness of the solution set}, for which we refer the reader e.g. to \cite[Chapter 1]{Aubin1984}, \cite[Chapter 10]{Aubin1990} or \cite[Chapter 2]{Vinter}. More precisely, Theorem \ref{thm:FilippovWass} extends Filippov's Theorem for Cauchy-Lipschitz differential inclusions in $(\Pcal_c(\R^d),W_p)$. For differential inclusions in a Banach space, this result is the most commonly used generalisation of the classical well-posedness theorems for Carathéodory ODEs. By construction, it also provides useful Gr\"onwall-type inequalities on the trajectory-selection pairs. In our context, the latter are based on non-trivial estimates for solutions of continuity equations which are presented in Proposition \ref{prop:MomentumGronwall} below. We proceed by proving in Theorem \ref{thm:RelaxationWass} a generalisation of the Relaxation theorem. Heuristically, this result asserts that any solution of a differential inclusion with a convexified righ-hand side can be approximated by solutions of the original differential inclusion. Relaxation results are very useful in cases where optimal trajectories may fail to exist, since they describe the closure of the solution set to differential inclusions (see e.g. \cite[Section 2.7]{Vinter}). In Theorem \ref{thm:Compactness}, we complement this result by showing that the set of solutions of a differential inclusion in $(\Pcal_c(\R^d),W_p)$ is compact in the topology of the uniform convergence, whenever the right-hand side of the inclusion is convex. This result is then applied in Theorem \ref{thm:Existence} to recover a general existence result for constrained mean-field optimal control problems.

\bigskip

The structure of the article is as follows. In Section \ref{section:Preliminaries}, we recall several notions pertaining to optimal transport theory, continuity equations, and set-valued analysis. In Section \ref{section:WassDiff}, we define differential inclusion in the Wasserstein space, and we prove our main results Theorem \ref{thm:FilippovWass}, Theorem \ref{thm:RelaxationWass} and Theorem \ref{thm:Compactness}. Finally in Section \ref{section:Existence}, we apply these new set-theoretic tools to show the existence of optimal controls for general fully non-linear and constrained mean-field optimal control problems. Appendix \ref{appendix:Proofs} contains the proofs of two new technical results, Lemma \ref{lem:OptimalMeasures} and Proposition \ref{prop:MomentumGronwall}. 


\section{Preliminaries}
\label{section:Preliminaries}

We introduce here all the necessary tools needed to formulate differential inclusions in Wasserstein spaces and prove our main results Theorem \ref{thm:FilippovWass}, Theorem \ref{thm:RelaxationWass} and Theorem \ref{thm:Compactness}.


\subsection{Analysis in measure spaces and optimal transport theory}
\label{subsection:MeasureTheory}

In this section, we recall some classical notations and results of measure theory and optimal transport. We refer the reader to the monographs \cite{AmbrosioFuscoPallara} and \cite{AGS,OTAM,villani1} respectively for a comprehensive introduction to these topics.

Let $(X,\Norm{\hspace{-0.025cm} \cdot \hspace{-0.025cm}}_X)$ be a separable Banach space. We denote by $\Pcal(X)$ the space of \textit{Borel probability measures} over $X$ endowed with the narrow topology, i.e. the coarsest topology such that the application
\begin{equation}
\label{eq:Narrow_convergence}
\mu \in \Pcal(X) \mapsto \INTDom{\phi(x)}{X}{\mu(x)} \in \R,
\end{equation}
is continuous for every $\phi \in C^0_b(X)$. Here $C^0_b(X)$ is the set of continuous and bounded real-valued functions over $X$, and we denote by $\mu_n \rightharpoonup^* \mu$ the  narrow convergence of measures induced by \eqref{eq:Narrow_convergence}. By Riesz's theorem (see e.g. \cite[Remark 5.1.2]{AGS}), the space $\Pcal(X)$ endowed with the narrow topology can be identified with a subset of the topological dual $(C^0_b(X))'$ of $(C^0_b(X),\Norm{\cdot}_{C^0})$, where $\Norm{\cdot}_{C^0}$ stands for the supremum norm. In the sequel given a metric space $(\Scal,d_{\Scal})$, we will use the notation $\AC([0,T],\Scal)$ for the space of absolutely continuous arcs with values in $\Scal$ and $\Lip(\phi(\cdot) \, ; \Omega)$ for the Lipschitz constant of a map $\phi : \Scal \rightarrow \R$ over $\Omega \subset \Scal$. 
 
Given two separable Banach spaces $(X,\Norm{\hspace{-0.06cm} \cdot \hspace{-0.06cm}}_X)$ and $(Y,\Norm{\hspace{-0.06cm} \cdot \hspace{-0.06cm}}_Y)$, an element $p \in [1,+\infty]$ and a measure $\mu \in \Pcal(X)$, we denote by $L^p(X,Y;\mu)$ and $W^{1,p}(X,Y;\mu)$ respectively the Lebesgue spaces of $p$-summable maps and Sobolev maps from $X$ into $Y$. For $p \in [1,+\infty)$, we also define the \textit{momentum of order $p$} of a measure $\mu \in \Pcal(X)$ as
\begin{equation*}
\M_p(\mu) := \left( \INTDom{|x|^p}{X}{\mu(x)} \right)^{1/p},
\end{equation*}
and consider the set $\Pcal_p(X) \subset\Pcal(X)$ of probability measures with finite momentum of order $p$, i.e. 
\begin{equation*}
\Pcal_p(X) = \Big\{ \mu \in \Pcal(X) ~\text{s.t.}~ \M_p(\mu) < +\infty \Big\}.
\end{equation*}
The \textit{support} of a probability measure $\mu \in \Pcal(X)$ is defined as the closed set
\begin{equation*}
\supp(\mu) = \Big\{ x \in X ~\text{s.t.}~ \mu(\Npazo_x) > 0 ~\text{for any neighbourhood $\Npazo_x$ of $x$} \Big\},
\end{equation*}
and we shall denote by $\Pcal_c(X) \subset \Pcal(X)$ the set of probability measures with compact support. 

We recall next the classical notion of \textit{pushforward} (or \textit{image measure}) of a Borel probability measure through a Borel map, along with that of \textit{transport plan}. 

\begin{Def}[Image of a measure through a Borel map] 
Given a measure $\mu \in \Pcal(X)$ and a Borel map $f : X \rightarrow Y$, the \textnormal{pushforward} $f_{\#} \mu$ of $\mu$ through $f(\cdot)$ is the unique Borel probability measure such that $f_{\#} \mu (B) = \mu(f^{-1}(B))$ for any Borel set $B \subset Y$. 
\end{Def}

\begin{Def}[Transport plans]
Given $\mu,\nu \in \Pcal(X)$, we say that $\gamma \in \Pcal(X^2)$ is a \textnormal{transport plan} between $\mu$ and $\nu$ -- denoted by $\gamma \in \Gamma(\mu,\nu)$ --, provided that
\begin{equation*}
\pi^1_{\#} \gamma = \mu \qquad \text{and} \qquad \pi^2_{\#} \gamma = \nu, 
\end{equation*}
where the maps $\pi^1,\pi^2 : X^2 \rightarrow X$ stand for the projection operators on the first and second factor.
\end{Def}

We now recall the definition and some of the main properties of the \textit{Wasserstein spaces} of optimal transport built over $X = \R^d$, for which we refer to \cite[Chapter 7]{AGS}, \cite[Chapter 5]{OTAM} or \cite[Chapter 6]{villani1}). 

\begin{Def}[Wasserstein spaces] 
\label{def:WassSpace}
Given $p \in [1,+\infty)$ and two probability measures $\mu,\nu \in \Pcal_p(\R^d)$, the \textnormal{Wasserstein distance of order $p$} between $\mu$ and $\nu$ is defined by
\begin{equation*}
W_p(\mu,\nu) = \min_{\gamma} \bigg\{ \bigg( \INTDom{|x-y|^p}{\R^{2d}}{\gamma(x,y)} \bigg)^{1/p} ~~ \text{s.t.} ~ \gamma \in \Gamma(\mu,\nu) \bigg\}.
\end{equation*}
The set of plans $\gamma \in \Gamma(\mu,\nu)$ achieving this minimum is denoted by $\Gamma_o(\mu,\nu)$ and referred to as the set of \textnormal{$p$-optimal transport plans} between $\mu$ and $\nu$. The space $(\Pcal_p(\R^d),W_p)$ of probability measures with finite $p$-th moment endowed with the $W_p$-metric is called the \textnormal{Wasserstein space} of order $p$.
\end{Def}

\begin{prop}[Properties of the Wasserstein spaces]
\label{prop:Properties_Wp}
For every $p \in [1,+\infty)$, the Wasserstein spaces $(\Pcal_p(\R^d),W_p)$ are complete and separable metric spaces. The topology induced by the $W_p$-metric metrises the narrow topology of probability measures induced by \eqref{eq:Narrow_convergence}, i.e.
\begin{equation*}
W_p(\mu_n,\mu) \underset{n \rightarrow +\infty}{\longrightarrow} 0 \qquad \text{if and only if} \qquad
\left\{
\begin{aligned}
\mu_n & \underset{n \rightarrow +\infty}{~ \rightharpoonup^*} \mu, \\
\INTDom{|x|^p}{\R^d}{\mu_n(x)} & \underset{n \rightarrow +\infty}{\longrightarrow} \INTDom{|x|^p}{\R^d}{\mu(x)}.
\end{aligned}
\right.
\end{equation*}
Given two measures $\mu,\nu \in \Pcal(\R^d)$, the Wasserstein distances are \textnormal{ordered}, i.e. $W_{p_1}(\mu,\nu) \leq W_{p_2}(\mu,\nu)$ whenever $p_1 \leq p_2$. Moreover when $p = 1$, the following \textnormal{Kantorovich-Rubinstein duality formula} holds 
\begin{equation} \label{eq:Kantorovich_duality}
W_1(\mu,\nu) = \sup_{\phi} \left\{ \INTDom{\phi(x) \,}{\R^d}{(\mu-\nu)(x)} ~\text{s.t.}~ \Lip(\phi \, ;\R^d) \leq 1 ~\right\}.
\end{equation}
\end{prop}


\subsection{Continuity equations in Wasserstein spaces}
\label{subsection:ContinuityEquations}

In this section, we recall some of the main definitions and classical results concerning continuity equations formulated in the space of measures. We also state several momentum and Gr\"onwall-type estimates needed in the proofs of two of our main results detailed in Section \ref{section:WassDiff} 

Let $T >0$ and $\Lcal^1$ be the standard one-dimensional Lebesgue measure on $[0,T]$. Throughout this paper, we shall always deal with \textit{Carathéodory vector fields}, that is mappings $v : [0,T] \times \R^d \rightarrow \R^d$ such that $t \mapsto v(t,x)$ is $\Lcal^1$-measurable for all $x \in \R^d$ and $x \mapsto v(t,x)$ is continuous for $\Lcal^1$-almost every $t \in [0,T]$. Moreover, we will always assume the following. 

\begin{hypsing}[\textbf{C1}]
There exists a map $m(\cdot) \in L^1([0,T],\R_+)$ such that 
\begin{equation*}
|v(t,x)| \leq m(t) \big( 1+|x| \big), 
\end{equation*}
for $\Lcal^1$-almost every $t \in [0,T]$ and all $x \in \R^d$. 
\end{hypsing}

We say that a curve of measures $\mu(\cdot) \in C^0([0,T],\Pcal(\R^d))$ solves a \textit{continuity equation} driven by a velocity field $v : [0,T] \times \R^d \rightarrow \R^d$ with initial condition $\mu^0 \in \Pcal(\R^d)$ provided that 
\begin{equation}
\label{eq:ContinuityEquation}
\left\{
\begin{aligned}
& \partial_t \mu(t) + \Div \big( v(t) \mu(t) \big) = 0, \\
& \mu(0) = \mu^0.
\end{aligned}
\right.
\end{equation}
This equation has to be understood in the sense of distributions against smooth and compactly supported test functions, i.e.
\begin{equation}
\label{eq:ContinuityDistrib1}
\INTSeg{\INTDom{\Big( \partial_t \phi(t,x) + \langle \nabla_x \phi(t,x) , v(t,x) \rangle \Big)}{\R^d}{\mu(t)(x)}}{t}{0}{T} = 0,
\end{equation}
for any $\phi \in C^{\infty}_c([0,T] \times \R^d)$. 

In our subsequent developments, we will deal with two notions of solution for \eqref{eq:ContinuityEquation}: \textit{superposition solutions} and \textit{Cauchy-Lipschitz solutions}. We shall henceforth denote by $\Sigma_T := C^0([0,T],\R^d)$ the space of continuous arcs in $\R^d$ and by $e_t : (x,\sigma) \in \R^d \times \Sigma_T \mapsto \sigma(t) \in \R^d$ the so-called \textit{evaluation map}. 

\begin{Def}[Superposition measures and solutions]
\label{def:Superposition_Measures}
We say that $\Beta \in \Pcal(\R^d \times \Sigma_T)$ is a \textnormal{superposition measure} generated by $v(\cdot,\cdot)$ if it is concentrated on the pairs $(x,\sigma) \in \R^d \times \AC([0,T],\R^d)$ such that
\begin{equation}
\label{eq:Characteristic_Def}
\sigma(0) = x \qquad \text{and} \qquad \dot \sigma(t) = v(t,\sigma(t)) ,
\end{equation}
for $\Lcal^1$-almost every $t \in [0,T]$. We further say that a distributional solution $\mu(\cdot) \in C^0([0,T],\Pcal(\R^d))$ of \eqref{eq:ContinuityEquation} is a \textnormal{superposition solution} if there exists a superposition measure $\Beta \in \Pcal(\R^d \times \Sigma_T)$ generated by $v(\cdot,\cdot)$ such that $\mu(t) = (e_t)_{\#} \Beta$ for all times $t \in [0,T]$. 
\end{Def}

One can easily check that if a superposition measure $\Beta \in \Pcal(\R^d \times \Sigma_T)$ generated by $v(\cdot,\cdot)$ satisfies the local integrability bounds
\begin{equation*}
\INTSeg{\INTDom{ \mathds{1}_K(\sigma(t))|v(t,\sigma(t))|}{\R^d \times \Sigma_T}{\Beta(x,\sigma)}}{t}{0}{T} < +\infty,
\end{equation*}
for any compact set $K \subset \R^d$, then the curve of measures $t \in [0,T] \mapsto \mu(t) = (e_t)_{\#} \Beta$ is a distributional solution of \eqref{eq:ContinuityEquation}. In the following theorem, we recall the converse of this statement which is known as the \textit{superposition principle}, and for which we refer the reader e.g. to \cite[Theorem 3.4]{AmbrosioC2014}. 

\begin{thm}[Superposition principle]
\label{thm:Superposition}
Let $\mu(\cdot) \in C^0([0,T],\Pcal(\R^d))$ be a distributional solution of \eqref{eq:ContinuityEquation} starting from $\mu^0 \in \Pcal_c(\R^d)$ and driven by a Carathéodory vector field $v(\cdot,\cdot)$ which satisfies \textbn{(C1)}. Then $\mu(\cdot)$ is a superposition solution, i.e. there exists a superposition measure $\Beta \in \Pcal(\R^d \times \Sigma_T)$ generated by $v(\cdot,\cdot)$ in the sense of Definition \ref{def:Superposition_Measures} such that $\mu(t) = (e_t)_{\#} \Beta$ for all times $t \in [0,T]$. 
\end{thm}

\begin{rmk}[On the statement of the superposition principle]
In \cite[Theorem 3.4]{AmbrosioC2014} and several other references in the literature (see e.g. \cite{Ambrosio2015,AGS}), the superposition principle is stated and proven for velocity fields $v(\cdot,\cdot)$ which are \textnormal{Borel measurable} with respect to both variables. We would like to point out that this result also holds true whenever $v(\cdot,\cdot)$ is only \textnormal{Lebesgue measurable} in its first variable.
\end{rmk}

The superposition principle is a very powerful tool which is used to prove many stability and existence results on continuity equations. In the next theorem, we recall several classical facts which state that under \textbf{(C1)} together with the Cauchy-Lipschitz regularity assumptions \textbf{(C2)} below, superposition solutions are unique and enjoy uniform boundedness and regularity properties. 

\begin{hypsing}[\textbf{C2}]
For any compact set $K \subset \R^d$, there exists a map $l_K(\cdot) \in L^1([0,T],\R_+)$ such that $\Lip(v(t,\cdot);K) \leq l_K(t)$ for $\Lcal^1$-almost every $t \in [0,T]$. 
\end{hypsing}

In the sequel, we use the notation $B(0,r)$ for the closed ball of radius $r >0$ centered at $0$ in $\R^d$ and $\Norm{\cdot}_1 \, :=\, \NormL{\cdot}{1}{[0,T]}$ for the $L^1$-norm of a real-valued function on the interval $[0,T]$.

\begin{thm}[Carathéodory and Cauchy-Lipschitz solutions of \eqref{eq:ContinuityEquation}]
\label{thm:CauchyLipschitz}
Let $r > 0$, $\mu^0 \in \Pcal(B(0,r))$ and $v : [0,T] \times \R^d \rightarrow \R^d$ be a Carathéodory vector field satisfying hypothesis \textbn{(C1)}. Then, there exists a curve of measures $\mu(\cdot) \in \AC([0,T],\Pcal_c(\R^d))$ solution of \eqref{eq:ContinuityEquation}, and every such solution curve satisfies
\begin{equation}
\label{eq:Uniform_BoundLip}
\supp(\mu(t)) \subseteq B(0,R_r) \qquad \text{and} \qquad W_p(\mu(t),\mu(s)) \leq \INTSeg{m_r(\tau)}{\tau}{s}{t},
\end{equation}
for all times $0 \leq s \leq t \leq T$ and any $p \in [1,+\infty)$, where 
\begin{equation*}
R_r = \Big( r~ + \Norm{m(\cdot)}_1 \hspace{-0.1cm} \Big) \exp \big( \hspace{-0.1cm} \Norm{m(\cdot)}_1 \hspace{-0.1cm} \big) \qquad \text{and} \qquad m_r(t) := (1+R_r)m(t),
\end{equation*} 
for $\Lcal^1$-almost every $t \in [0,T]$. Furthermore if the velocity field $v(\cdot,\cdot)$ also satisfies hypothesis \textbn{(C2)}, the solutions of \eqref{eq:ContinuityEquation} are unique.
\end{thm}

\begin{proof}
The proof of Theorem \ref{thm:CauchyLipschitz} follows from Theorem \ref{thm:Superposition}, together with classical boundedness and absolute continuity estimates on the flows of Carathéodory ODEs in the Cauchy-Lipschitz framework. For further details, see \cite[Section 3]{AmbrosioC2014}.
\end{proof}

We end this section by several new estimates that will be useful in the sequel. In Proposition \ref{prop:MomentumGronwall}, we present new momentum and Gr\"onwall-type estimates, which to the best of our knowledge are not written at this degree of generality in the literature. The proof of these results along with those of several crucial steps of Theorem \ref{thm:RelaxationWass} are based on Lemma \ref{lem:OptimalMeasures} below. The latter is inspired by \cite[Section 2.2]{UsersGuidOT}, and provides the existence of a transport plan $\hat{\Beta}_{\mu,\nu} \in \Gamma(\Beta_{\mu},\Beta_{\nu})$, whose pushforwards through the evaluation maps $(e_t,e_t)$ are optimal transport plans between $\mu(t)$ and $\nu(t)$ for all times $t \in [0,T]$. 

The arguments subtending these two results being somewhat technical, we postpone their proofs to Appendix \ref{appendix:Proofs}. 

\begin{lem}[Superposition measures producing optimal plans]
\label{lem:OptimalMeasures}
Let $K \subset \R^d$ be a compact set and $\mu(\cdot),\nu(\cdot) \in \AC([0,T],\Pcal(K))$ be two solutions of \eqref{eq:ContinuityEquation} driven respectively by velocity fields $v,w : [0,T] \times \R^d \rightarrow \R^d$ satisfying hypothesis \textbn{(C1)}. Let $\Beta_{\mu},\Beta_{\nu} \in \Pcal(\R^d \times \Sigma_T)$ be two superposition measures given by Theorem \ref{thm:Superposition}, i.e. $\mu(t) = (e_t)_{\#} \Beta_{\mu}$ and $\nu(t) = (e_t)_{\#} \Beta_{\nu}$ for all times $t \in [0,T]$.
 
Then for any $p \in [1,+\infty)$, there exists a transport plan $\hat{\Beta}_{\mu,\nu} \in \Gamma(\Beta_{\mu},\Beta_{\nu})$ such that for all times $t \in [0,T]$, it holds 
\begin{equation}
\label{eq:OptimalPlan_Statement}
(\pi_{\R^d} ,\pi_{\R^d} )_{\#} \hat{\Beta}_{\mu,\nu} \in \Gamma_o(\mu^0,\nu^0) \qquad \text{and} \qquad (e_t,e_t)_{\#} \hat{\Beta}_{\mu,\nu} \in \Gamma_o(\mu(t),\nu(t)), 
\end{equation}
where $\Gamma_o$ stands for the set of $p$-optimal transport plans between two measures.
\end{lem}

\begin{prop}[Momentum and Gr\"onwall estimates]
\label{prop:MomentumGronwall} 
Let $K \subset \R^d$ be a compact set and $\mu(\cdot),\nu(\cdot) \in \AC([0,T],\Pcal(K))$ be two solutions of \eqref{eq:ContinuityEquation} driven respectively by velocity fields $v,w : [0,T] \times \R^d \rightarrow \R^d$ satisfying hypothesis \textbn{(C1)}. Furthermore, suppose that $v(\cdot,\cdot)$ also satisfies hypothesis \textbn{(C2)}. 

Then for any $p \in [1,+\infty)$, we have 
\begin{equation}
\label{eq:MomentumEstimate}
\M_p(\mu(t)) \leq \, C_p \left( \M_p(\mu^0) + \INTSeg{m(s)}{s}{0}{t} \right) \exp \left( C_p' \NormL{m(\cdot)}{1}{[0,t]}^p \right), 
\end{equation}
and 
\begin{equation}
\label{eq:WassEstimate}
\begin{aligned}
W_p(\mu(t),\nu(t)) \leq \, C_p \left( W_p(\mu(0),\nu(0)) + \INTSeg{\NormC{v(s,\cdot)-w(s,\cdot)}{0}{K,\R^d}}{s}{0}{t} \right) \exp \left( C'_p \NormL{l_K(\cdot)}{1}{[0,t]}^p \right)&, 
\end{aligned}
\end{equation}
for all times $t \in [0,T]$, where the constants $C_p,C_p' > 0$ are defined by 
\begin{equation}
\label{eq:ConstExpression}
C_p = 2^{(p-1)/p} \qquad \text{and} \qquad C'_p = \tfrac{2^{p-1}}{p}.
\end{equation}
\end{prop}

\begin{rmk}[A more general momentum estimate]
\label{rmk:GeneralisedMomentum}
In the proof of Theorem \ref{thm:FilippovWass} below, we will need a more general variant of \eqref{eq:MomentumEstimate} for velocity fields $v : [0,T] \times \R^d \rightarrow \R^d$ satisfying the sub-linearity estimate 
\begin{equation*}
|v(t,x)| \leq m(t) \Big( 1 + |x| + M(t) \Big), 
\end{equation*}
for $\Lcal^1$-almost every $t \in [0,T]$ and all $x \in \R^d$, where $m(\cdot) \in L^1([0,T],\R_+)$ and $M(\cdot) \in L^{\infty}([0,T],\R_+)$. In this case, one can show by repeating the arguments of the proof of \eqref{eq:MomentumEstimate} in Appendix \ref{appendix:Proofs} that 
\begin{equation}
\label{eq:GeneralisedMomentum}
\M_p(\mu(t)) \leq \, C_p \left( \M_p(\mu^0) + \INTSeg{m(s) \Big( 1+M(s) \Big)}{s}{0}{t} \right) \exp \left( C_p' \NormL{m(\cdot)}{1}{[0,t]}^p \right), 
\end{equation}
for all times $t \in [0,T]$, and that \eqref{eq:WassEstimate} is still verified.
\end{rmk}


\subsection{Elements of set-valued analysis}
\label{subsection:SetValued}

In this section, we recall some notations and basic notions pertaining to set-valued analysis and multifunctions. We refer the reader to \cite{Aubin1990} for most of the results which are stated below, as well as for a general treatment of this topic.

Let $(\Scal,d_{\Scal})$ be a complete separable metric space and $(X,\Norm{\cdot}_X)$ be a separable Banach space. In the sequel, we denote by $\B_{\Scal}(s,\beta)$ the closed ball of center $s \in \Scal$ and radius $\beta >0$ and by $\B_X$ the closed unit ball in $X$. Given a subset $B \subset X$, we denote by $\co \hspace{-0.05cm}(B)$ its \textit{closed convex-hull} defined as the closure in $X$ of the set of convex combinations
\begin{equation*}
\bigcup_{N \geq 1} \bigg\{ \mathsmaller{\sum}\limits_{i=1}^N \alpha_i b_i ~\text{s.t.}~ b_i \in B,~ \alpha_i \geq 0 ~\text{for $i \in\{1,\dots,N\}$ and}~ \mathsmaller{\sum}\limits_{i=1}^N \alpha_i = 1 \bigg\}.
\end{equation*}
We say that an application $\F : \Scal \rightrightarrows X$ is a \textit{set-valued map} -- or a \textit{multifunction} -- from $\Scal$ into $X$ if $\F(s) \subset X$ for all $s \in \Scal$, and we define its \textit{domain} by $\textnormal{dom}(\F) := \left\{ s \in \Scal ~\text{s.t.}~ \F(s) \neq \emptyset \right\}$. A multifunction $\F(\cdot)$ has \textit{closed values} if $\F(s)$ is closed in $X$ for any $s \in \textnormal{dom}(\F)$, and we say that $\F(\cdot)$ itself is \textit{closed} if its graph
\begin{equation}
\label{eq:Graph}
\Graph(\F) := \Big\{(s,x) ~\text{s.t.}~ x \in \F(s) \Big\}, 
\end{equation}
is a closed subset of $\Scal \times X$.

\begin{Def}[Measurable set-valued maps and measurable selections]
\label{def:MeasMultifunction}
We say that a set-valued map $\F : [0,T] \rightrightarrows X$ with closed images is $\Lcal^1$-\textnormal{measurable} if the sets
\begin{equation*}
\F^{-1}(\Opazo) := \Big\{ t \in [0,T] ~\text{s.t.}~ \F(t) \cap \Opazo \neq \emptyset \Big\},
\end{equation*}
are $\Lcal^1$-measurable for any open set $\Opazo \subset X$. A single-valued map $f : [0,T] \rightarrow X$ is called a \textnormal{measurable selection} of $\F(\cdot)$ if it is $\Lcal^1$-measurable and such that $f(t) \in \F(t)$ for $\Lcal^1$-almost every  $t \in [0,T]$. 
\end{Def}

\begin{Def}[Lipschitz regularity of set-valued maps]
\label{def:LipSetValued}
We say that a set-valued map $\F : \Scal \rightrightarrows X$ is \textnormal{$L$-Lipschitz} around $s \in \Scal$ if there exists a neighbourhood $\Npazo_s \subset \textnormal{dom}(\F)$ of $s$ such that 
\begin{equation*}
\F(s_1) \subset \F(s_2) + L d_{\Scal}(s_1,s_2) \, \B_X, 
\end{equation*}
for every $s_1,s_2 \in \Npazo_s$.
\end{Def}

In the following theorem, we recall a classical measurable selection principle for set-valued maps, see for instance \cite[Theorem 8.1.3]{Aubin1990}.

\begin{thm}[Measurable selection]
\label{thm:Measurable}
Let $\F : [0,T] \rightrightarrows X$ be an $\Lcal^1$-measurable set-valued map with non-empty and closed images. Then, $\F(\cdot)$ admits a measurable selection. 
\end{thm}

Measurable selections are crucial for investigating solutions of differential inclusions and will therefore appear frequently in the proofs of our main results. However, we will primarily work with set-valued maps $\F(\cdot)$ with values in the linear space $C^0(\R^d,\R^d)$ which is not a Banach space. To circumvent this difficulty, we introduce in the following definition a concept of measurability adapted to this type of multifunction along with a suitable notion of closed convex hull for subsets of $C^0(\R^d,\R^d)$.

\begin{Def}[Compact restrictions, measurability and convex hulls]
\label{def:CompactRestriction}
Let $\F : [0,T] \rightrightarrows C^0(\R^d,\R^d)$ be a set-valued map. For every compact set $K \subset \R^d$, we define the \textnormal{compact restriction} $\F_K : [0,T] \rightrightarrows C^0(K,\R^d)$ of $\F(\cdot)$ to $K$ as
\begin{equation*}
\F_K(t) := \Big\{ f_{\vert K} \in C^0(K,\R^d) ~\text{s.t.}~ f(t) \in \F(t) \Big\},
\end{equation*}
for $\Lcal^1$-almost every $t \in [0,T]$, where $f_{\vert K}$ denotes the restriction of the map $f \in C^0(\R^d,\R^d)$ to $K$. We then say that $\F : [0,T] \rightrightarrows C^0(\R^d,\R^d)$ is $\Lcal^1$-measurable provided that $\F_K(\cdot)$ is $\Lcal^1$-measurable for every compact set $K \subset \R^d$.

Analogously given a set $\F \subset C^0(\R^d,\R^d)$, we define its closed convex hull $\co \F \subset C^0(\R^d,\R^d)$ by
\begin{equation}
\label{eq:ConvbarDef}
\co \F := \Big\{ f \in C^0(\R^d,\R^d) ~\text{s.t.}~ f_{\vert K} \in \co\F_K ~\text{for every compact set $K \subset \R^d$} \Big\},
\end{equation}
where $\co \F_{K}$ is taken in the Banach space $(C^0(K,\R^d),\Norm{\cdot}_{C^0})$.
\end{Def}

\begin{rmk}[Link with the Whitney topology]
Even though $C^0(\R^d,\R^d)$ is not a Banach space, it can be endowed with the so-called \textnormal{Whitney topology} (see e.g. \cite[Chapter 1]{Hirsch1974}), where open sets are defined using bases of neighbourhoods involving localisations on compact sets. In this context, the notion of measurability introduced in Definition \ref{def:CompactRestriction} coincides with that of measurability taken with respect to the Whitney topology in $C^0(\R^d,\R^d)$.
\end{rmk}

We state in the following lemma two measurability results which are inspired by \cite[Section 1]{Frankowska1990}, see also \cite[Chapter 8]{Aubin1990} for more general statements.

\begin{lem}[Some measurability results]
\label{lem:Measurable}
Let $\F : [0,T] \rightrightarrows X$ and $\G : [0,T] \times \Scal \rightrightarrows X$ be set-valued maps with closed non-empty images, and also $f : [0,T] \rightarrow X$ and $L : [0,T] \rightarrow \R_+$ be $\Lcal^1$-measurable single-valued maps. Then, the following statements hold true. 
\begin{enumerate}
\item[\textnormal{(a)}] If $\F(\cdot)$ is $\Lcal^1$-measurable and the set-valued map
\begin{equation*}
\Hpazo : t \in [0,T] \rightrightarrows \F(t) \cap \big\{ f(t) + L(t) \B_X \big\} \subset X,
\end{equation*}
has non-empty images for $\Lcal^1$-almost every $t \in [0,T]$, then $\Hpazo(\cdot)$ admits a measurable selection. 
\item[\textnormal{(b)}] Let $s(\cdot) \in C^0([0,T],\Scal)$ and $\beta >0$ be a constant such that for $\Lcal^1$-almost every $t \in [0,T]$, the multifunction $\G(t,\cdot)$ is $L(t)$-Lipschitz over $\B_{\Scal}(s(t),\beta)$. Then, the set-valued map
\begin{equation*}
t \in [0,T] \rightrightarrows \G(t,s(t)) \subset X, 
\end{equation*}
is $\Lcal^1$-measurable.
\end{enumerate}
\end{lem}

\begin{proof}
Statement (a) follows from \cite[Corollary 8.2.13 and Theorem 8.2.4]{Aubin1990}. The same arguments as those of the proof of \cite[Theorem 8.2.8]{Aubin1990} yield statement (b). 
\end{proof}

We end these prerequisites by recalling the notion of \textit{Aumann integral} for set-valued maps, and a generalisation of Aumann's theorem on the closure of the integral which is taken from \cite[Theorem 8.6.4]{Aubin1990}. In the sequel, all the integrals are taken in the sense of Bochner (see e.g. \cite[Chapter 2]{DiestelUhl}). 

\begin{Def}[Integrably bounded multifunctions and Aumann integral]
\label{def:Aumann}
An $\Lcal^1$-measurable set-valued map $\F : [0,T] \rightrightarrows X$ with closed images is said to be \textnormal{integrably bounded} if there exists $k(\cdot) \in L^1([0,T],\R_+)$ such that for $\Lcal^1$-almost every $t \in [0,T]$, it holds
\begin{equation*}
\F(t) \subset k(t) \B_X.
\end{equation*}
The \textnormal{Aumann integral} of $\F(\cdot)$ over a measurable subset $\Omega \subset [0,T]$ is then defined by 
\begin{equation*}
\INTDom{\F(t)}{\Omega}{t} ~:=~ \left\{ \INTDom{f(t)}{\Omega}{t} ~\text{s.t.}~ f(t) \in \F(t) ~ \text{for $\Lcal^1$-almost every $t \in \Omega$} \right\}. 
\end{equation*}
\end{Def}

\begin{lem}[Closure of the Aumann integral]
\label{lem:Aumann}
Let $\F : [0,T] \rightrightarrows X$ be an $\Lcal^1$-measurable and integrably bounded set-valued map. Then for any measurable set $\Omega \subset [0,T]$, it holds 
\begin{equation*}
\INTDom{\co \F(t)}{\Omega}{t} = \overline{\INTDom{\F(t)}{\Omega}{t}}.
\end{equation*}
In particular, for any measurable selection $f(\cdot)$ from $\co \F(\cdot)$ and any $\delta > 0$, there exists a measurable selection $f_{\delta}(\cdot)$ in $\F(\cdot)$ such that 
\begin{equation*}
\left\| \, \INTDom{f(t)}{\Omega}{t} - \INTDom{f_{\delta}(t) \,}{\Omega}{t} \, \right\|_X \,\leq \delta. 
\end{equation*} 
\end{lem}


\section{Differential Inclusions in Wasserstein Spaces}
\label{section:WassDiff}

In this section, we move on to the main object of this article which are differential inclusions in Wasserstein spaces. 


\subsection{Definition of differential inclusions in Wasserstein spaces}
\label{subsection:DefInc}

In what follows, we state our definition of differential inclusions in Wasserstein spaces and study some of the main properties of their solutions. We recall that our intention is to build an adequate set-valued generalisation of non-local continuity equations of the form
\begin{equation}
\label{eq:ContinuityEquationBis}
\partial_t \mu(t) + \Div \big( v(t,\mu(t)) \mu(t) \big) = 0. 
\end{equation}
Motivated by the discussion presented in the Introduction, we propose the following  notion of differential inclusion in Wasserstein spaces inspired by \cite[Chapter 10]{Aubin1990} and \cite[Section 3]{Gangbo2011}. 

\begin{Def}[Differential inclusions in Wasserstein spaces]
\label{def:WassInc}
Let $T > 0$ and $V : (t,\mu) \in [0,T] \times \Pcal_c(\R^d) \rightrightarrows C^0(\R^d,\R^d)$ be a set-valued map. We say that a curve of measures $\mu(\cdot) \in \AC([0,T],\Pcal_c(\R^d))$ is a solution of the \textnormal{Wasserstein differential inclusion}
\begin{equation}
\label{eq:DefWassInc}
\partial_t \mu(t) \in - \Div \Big( V(t,\mu(t)) \mu(t) \Big),
\end{equation}
if there exists a measurable selection $t \in [0,T] \mapsto v(t) \in V(t,\mu(t))$ in the sense of Definition \ref{def:CompactRestriction} such that the \textnormal{trajectory-selection pair} $(\mu(\cdot),v(\cdot))$ solves the continuity equation
\begin{equation*}
\partial_t \mu(t) + \Div \big( v(t) \mu(t) \big) = 0, 
\end{equation*}
in the sense of distributions. 
\end{Def}

We are now ready to state our working assumptions for the rest of this section. From now on, we fix a time horizon $T > 0$ and a real number $p \in [1,+\infty)$.

\begin{hyp}[\textbf{DI}]
For every $R > 0$, assume that the following holds with $K := B(0,R)$.
\begin{enumerate}
\item[$(i)$] For any $\mu \in \Pcal_c(\R^d)$, the set-valued map $t \in [0,T] \rightrightarrows V_K(t,\mu)$ is $\Lcal^1$-measurable with closed non-empty images in $C^0(K,\R^d)$.
\item[$(ii)$] There exists a map $m(\cdot) \in L^1([0,T],\R_+)$ such that for $\Lcal^1$-almost every $t \in [0,T]$, for any $\mu \in \Pcal_c(\R^d)$, for every $v \in V(t,\mu)$ and all $x \in \R^d$, it holds 
\begin{equation*}
|v(x)| \leq m(t) \Big(1+|x|+\M_1(\mu) \Big).
\end{equation*}
\item[$(iii)$] There exists a map $l_K(\cdot) \in L^1([0,T],\R_+)$ such that for $\Lcal^1$-almost every $t \in [0,T]$, any $\mu \in \Pcal(K)$ and every $v \in V(t,\mu)$, it holds 
\begin{equation*}
\Lip \big(v(\cdot) \,;K \big) \leq l_K(t).
\end{equation*} 
\item[$(iv)$] There exists a map $L_K(\cdot) \in L^1([0,T],\R_+)$ such that for $\Lcal^1$-almost every $t \in [0,T]$ and any $\mu,\nu \in \Pcal(K)$, it holds 
\begin{equation*}
V_K(t,\nu) \subset V_K(t,\mu) + L_K(t) W_p(\mu,\nu) \B_{C^0(K,\R^d)}.
\end{equation*} 
\end{enumerate}
\end{hyp}

These assumptions are a localised extension of those e.g. from \cite{Pedestrian}, and seem rather minimal for the elaboration of a Cauchy-Lipschitz theory of differential inclusions. We end these preliminaries by providing support and regularity estimates on the solutions of \eqref{eq:DefWassInc}.

\begin{prop}[Support and regularity estimates on solutions of \eqref{eq:DefWassInc}]
\label{prop:IncSuppBound}
Let $V : [0,T] \times \Pcal_c(\R^d) \rightrightarrows C^0(\R^d,\R^d)$ be a set-valued map satisfying hypothesis \textbn{(DI)}-$(ii)$. Then for any $r >0$, there exists a constant $R_r > 0$ and a map $m_r(\cdot) \in L^1([0,T],\R_+)$ such that any solution $\mu(\cdot) \in \AC([0,T],\Pcal_c(\R^d))$ of \eqref{eq:DefWassInc} starting from $\mu^0 \in \Pcal(B(0,r))$ satisfies 
\begin{equation}
\label{eq:Prop_SuppLip}
\supp(\mu(t)) \subset K := B(0,R_r) \qquad \text{and} \qquad W_p(\mu(t),\mu(s)) \leq \INTSeg{m_r(\tau)}{\tau}{s}{t},
\end{equation}
for all times $0 \leq s \leq t \leq T$ and any $p \in [1,+\infty)$.
\end{prop}

\begin{proof}
Let $\mu(\cdot) \in \AC([0,T],\Pcal_c(\R^d))$ be a solution of \eqref{eq:DefWassInc} starting from $\mu^0 \in \Pcal(B(0,r))$. By Definition \ref{def:WassInc}, there exists a measurable selection $t \in [0,T] \mapsto v(t) \in V(t,\mu(t))$ such that  
\begin{equation*}
\left\{
\begin{aligned}
& \partial_t \mu(t) + \Div \big( v(t) \mu(t) \big) = 0, \\
& \mu(0) = \mu^0. 
\end{aligned}
\right.
\end{equation*}
By construction, $v(\cdot,\cdot)$ is a Carathéodory vector field which satisfies the sub-linearity estimate 
\begin{equation}
\label{eq:SublinearityMomentum_Proof}
|v(t,x)| \leq m(t) \Big( 1 + |x| + \M_1(\mu(t)) \Big), 
\end{equation}
for $\Lcal^1$-almost every $t \in [0,T]$ and all $x \in \R^d$ as a consequence of hypothesis \textbn{(DI)}-$(ii)$. By repeating the argument e.g. of \cite[Lemma 5.1]{FLOS}, one can show that  
\begin{equation*}
\M_1(\mu(t)) \leq \left( r + \INTSeg{m(s)}{s}{0}{t} \right) \exp \left( 2 \NormL{m(\cdot)}{1}{[0,t]} \right), 
\end{equation*}
for all times $t \in [0,T]$. Therefore, the velocity field $v(\cdot,\cdot)$ satisfies the uniform sub-linearity estimate 
\begin{equation*}
|v(t,x)| \leq (1+M_r) m(t) \big( 1 + |x| \big), 
\end{equation*}
where $M_r := \left( r\, + \Norm{m(\cdot)}_1 \right) \exp \left( 2 \Norm{m(\cdot)}_1 \right)$, and by Theorem \ref{thm:CauchyLipschitz} there exists a constant $R_r > 0$ and a map $m_r(\cdot) \in L^1([0,T],\R_+)$ such that \eqref{eq:Prop_SuppLip} holds.
\end{proof}


\subsection{Filippov Theorem}
\label{subsection:Filippov}

In this section, we state and prove a natural generalisation of Filippov's Theorem for Wasserstein differential inclusions formulated in $(\Pcal_c(\R^d),W_p)$.

\begin{thm}
\label{thm:FilippovWass}
Let $T > 0$, $p \in [1,+\infty)$ and $V : [0,T] \times \Pcal_c(\R^d) \rightrightarrows C^0(\R^d,\R^d)$ be a  set-valued map satisfying hypotheses \textbn{(DI)}. Let $\nu(\cdot) \in \AC([0,T],\Pcal(K_{\nu}))$ be a solution of \eqref{eq:ContinuityEquation} induced by a velocity field $w : [0,T] \times \R^d \rightarrow \R^d$ satisfying \textbn{(C1)}, where $K_{\nu}:= B(0,R_{\nu})$ is a closed ball. Furthermore, suppose that the \textnormal{mismatch  function} $\eta_{\nu} : [0,T] \rightarrow \R_+$ defined by 
\begin{equation}
\label{eq:Mismatch}
\eta_{\nu}(t) := \dist_{C^0(K_{\nu},\R^d)} \Big( w_{\vert K_{\nu}}(t),V_{K_{\nu}}(t,\nu(t)) \Big),
\end{equation}
for $\Lcal^1$-almost every $t \in [0,T]$ is integrable over $[0,T]$. 

Then for any $r > 0$ and every measure $\mu^0 \in \Pcal(B(0,r))$, there exists $K := B(0,R_r)$ depending only on the magnitude of $r,R_{\nu},T,\Norm{m(\cdot)}_1$ and a trajectory-selection pair $(\mu(\cdot),v(\cdot)) \in \AC([0,T],\Pcal(K)) \times L^1([0,T],C^0(K,\R^d))$ solution of the inclusion 
\begin{equation}
\label{eq:Theorem_DiffInc}
\left\{
\begin{aligned}
& \partial_t \mu(t) \in - \Div \Big( V(t,\mu(t)) \mu(t) \Big), \\ 
& \mu(0) = \mu^0,
\end{aligned}
\right. 
\end{equation}
such that 
\begin{equation}
\label{eq:Theorem_DistEst}
W_p(\mu(t),\nu(t)) \leq \chi_p(t) \, \exp \big( C_{K,p}(t) \big), 
\end{equation}
for all times $t \in [0,T]$. Moreover, it holds for $\Lcal^1$-almost every $t \in [0,T]$ that  
\begin{equation}
\label{eq:Theorem_VelEst}
\NormC{v(t,\cdot) - \hat{w}(t,\cdot)}{0}{K,\R^d} \leq L_K(t) \, \chi_p(t) \exp \big( C_{K,p}(t) \big) + \eta_{\nu}(t),
\end{equation}
for $\Lcal^1$-almost every $t \in [0,T]$, where $\hat{w}(t,x) := w(t,\pi_{K_{\nu}}(x))$ with $\pi_{K_{\nu}}$ the projection onto $K_{\nu}$, and
\begin{equation}
\label{eq:ChiCkp_Def}
\left\{
\begin{aligned}
\chi_p(t) & := C_p \left( W_p(\mu^0,\nu(0)) + \INTSeg{\eta_{\nu}(s)}{s}{0}{t} \right) \exp \left( C_p' \NormL{l_K(\cdot)}{1}{[0,t]}^p \right), \\
C_{K,p}(t) & := C_p \left( \INTSeg{L_K(s)}{s}{0}{t} \right) \exp \left( C_p' \NormL{l_K(\cdot)}{1}{[0,t]}^p \right),
\end{aligned}
\right.
\end{equation}
for all times $t \in [0,T]$, the constants $C_p,C_p'>0$ being defined as in \eqref{eq:ConstExpression}
\end{thm}

Before moving on to the proof of Theorem \ref{thm:FilippovWass}, we state a technical lemma dealing with chained integral estimates and which proof is the matter of an elementary induction argument. 

\begin{lem}[A uniform estimate for chained integral inequalities]
\label{lem:IntegralIneq}
Let $m(\cdot) \in L^1([0,T],\R_+)$ and $\alpha>0$ be a constant. Let $N \in \N \cup \{+\infty\}$ and $(f_n(\cdot))_{0 \leq n < N} \subset C^0([0,T],\R_+)$ be a family of maps such that 
\begin{equation}
\label{eq:Induction_Inequality}
f_{n+1}(t) \leq \alpha \left( 1 + \INTSeg{m(s) f_n(s)}{s}{0}{t} \right), 
\end{equation}
for all $t \in [0,T]$ and every $n \geq 0$ such that $n+1 <N$. Then, there exists a constant $C > 0$ depending only on the magnitudes of $\NormC{f_0(\cdot)}{0}{[0,T]},\alpha,T$ and $\Norm{m(\cdot)}_1$ such that
\begin{equation}
\label{eq:Linfinity_Bound}
\sup_{0 \leq n <N}\NormC{f_n(\cdot)}{0}{[0,T]} \leq C.
\end{equation}
\end{lem}

The proof of Theorem \ref{thm:FilippovWass} is based on an iterative scheme ``à la Picard'' in the spirit of the classical proof of Filippov's Theorem (see e.g. \cite[Theorem 2.3.13]{Vinter}), and is split into four steps. In Step 1, we initialise our sequence of approximations and list some of the properties of its first element. We then show in Step 2 how the whole sequence of trajectory-selection pairs can be built, and prove in Step 3 that it is a Cauchy sequence in $C^0([0,T],\Pcal(K)) \times L^1([0,T],C^0(K,\R^d))$. Finally in Step 4, we show that the corresponding limit trajectory-selection pair is a solution of the differential inclusion \eqref{eq:Theorem_DiffInc} satisfying \eqref{eq:Theorem_DistEst} and \eqref{eq:Theorem_VelEst}.

\begin{proof}[Proof (of Theorem \ref{thm:FilippovWass})]
To  simplify the computations, we will restrict our attention to the case $p=1$. The general case is similar and can be recovered from minor variations of the following arguments. 

Fix $r > 0$ and $\mu^0 \in \Pcal(B(0,r))$. Our goal is to build a closed ball $K := B(0,R_r) \subset \R^d$ along with a sequence of pairs $(\mu_n(\cdot),v_n(\cdot)) \subset \AC([0,T],\Pcal(K)) \times L^1([0,T],C^0(K,\R^d))$ solutions of
\begin{equation}
\label{eq:Construction1}
\left\{
\begin{aligned}
& \partial_t \mu_n(t) + \Div (v_n(t) \mu_n(t)) = 0, \\
& \mu_n(0) = \mu^0,
\end{aligned}
\right.
\end{equation}
which satisfy
\begin{equation}
\label{eq:Construction2}
\left\{
\begin{aligned}
& v_{n+1}(t) \in V(t,\mu_n(t)), ~~ & \\
& \hspace{-0.1cm} \NormC{v_{n+1}(t,\cdot)-v_n(t,\cdot)}{0}{K,\R^d} \leq L_K(t) W_1(\mu_n(t),\mu_{n-1}(t)), &   \\
& \hspace{-0.1cm} \M_1(\mu_n(t)) \leq C, &  
\end{aligned}
\right.
\end{equation}
for $\Lcal^1$-almost every $t \in [0,T]$ and all $n \geq 1$, where $C > 0$ is a uniform constant. 


\paragraph*{Step 1: Initialisation of the sequence.}

We set $\mu^0(\cdot) := \nu(\cdot)$ and $v_0(\cdot,\cdot) := \hat{w}(\cdot,\cdot)$. Remark first that as a consequence of hypotheses \textbf{(DI)}-$(i)$, \textbf{(DI)}-$(iv)$ and Lemma \ref{lem:Measurable}-(b), the set-valued map 
\begin{equation*}
t \in [0,T] \rightrightarrows V_{K_{\nu}}(t,\nu(t)) \subset C^0(K_{\nu},\R^d),
\end{equation*}
is $\Lcal^1$-measurable. Moreover as a by-product of hypotheses \textbf{(DI)}-$(ii)$, \textbf{(DI)}-$(iii)$ and of the Ascoli-Arzelà Theorem, the sets $V_{K_{\nu}}(t,\nu(t))$ are compact in the separable Banach space $C^0(K_{\nu},\R^d)$ for $\Lcal^1$-almost every $t \in [0,T]$. Whence, the multifunction
\begin{equation*}
t \in [0,T] \rightrightarrows V_{K_{\nu}}(t,\nu(t)) \cap \Big\{ w_{\vert K_{\nu}}(t) + \eta_{\nu}(t) \B_{C^0(K_{\nu},\R^d)} \Big\},
\end{equation*}
is measurable with compact and non-empty images for $\Lcal^1$-almost every $t \in [0,T]$. We can thus apply Lemma \ref{lem:Measurable}-(a) to find a measurable selection $t \in [0,T] \mapsto v_1(t) \in V_{K_{\nu}}(t,\nu(t))$ such that 
\begin{equation}
\label{eq:Ineq_eta}
\Norm{v_1(t,\cdot) - w_{\vert K_{\nu}}(t,\cdot)}_{C^0(K_{\nu},\R^d)} \hspace{0.1cm} = \hspace{0.1cm} \eta_{\nu}(t), 
\end{equation}
for $\Lcal^1$-almost every $t \in [0,T]$.

Since $v_1(t) \in V_{K_{\nu}}(t,\nu(t))$ for $\Lcal^1$-almost every $t \in [0,T]$, one has by \textbf{(DI)}-$(ii)$ and \textbf{(DI)}-$(iii)$ 
\begin{equation}
\label{eq:Sublinearity_u1}
|v_1(t,x)| \leq m(t) \Big( 1 + |x| + \M_1(\nu(t)) \Big) \qquad \text{and} \qquad \qquad \Lip(v_1(t,\cdot) \,; K_{\nu}) \leq l_{K_{\nu}}(t), 
\end{equation}
for $\Lcal^1$-almost every $t \in [0,T]$ and any $x \in K_{\nu}$. Let $\pi_{K_{\nu}} : \R^d \rightarrow K_{\nu}$ denote the projection operator onto $K_{\nu}$, and observe that 
\begin{equation}
\label{eq:ProjectionProperties}
|\pi_{K_{\nu}}(x)| \leq |x| \qquad \text{and} \qquad |\pi_{K_{\nu}}(y) - \pi_{K_{\nu}}(x)| \leq |y-x|, 
\end{equation}
for any $x,y \in \R^d$. We can therefore define an extension\footnote{For simplicity, we use the same notation for the original vector field and for its extension.} of $v_1(t,\cdot)$ to $\R^d$ by
\begin{equation*}
v_1(t,\cdot) : x \in \R^d \mapsto v_1 \big( t,\pi_{K_{\nu}}(x) \big),
\end{equation*}
for $\Lcal^1$-almost every $t \in [0,T]$. One can easily check using \eqref{eq:ProjectionProperties} that this extension satisfies
\begin{equation*}
|v_1(t,x)| \leq m(t) \Big( 1 + |x| + \M_1(\nu(t)) \Big) \qquad \text{and} \qquad \Lip(v_1(t,\cdot)\, ;\R^d) \leq l_{K_{\nu}}(t),   
\end{equation*}
for $\Lcal^1$-almost every $t \in [0,T]$ and all $x \in \R^d$. Whence, the extended velocity field  $v_1(\cdot,\cdot)$ satisfies the hypotheses \textbn{(C1)}-\textbn{(C2)} of Theorem \ref{thm:CauchyLipschitz}, and it induces a unique solution $\mu_1(\cdot) \in \AC([0,T],\Pcal_c(\R^d))$ of
\begin{equation}
\label{eq:InitialStep_Cauchy}
\left\{
\begin{aligned}
& \partial_t \mu_1(t) + \Div (v_1(t) \mu_1(t)) = 0, \\
& \mu_1(0) = \mu^0. 
\end{aligned}
\right.
\end{equation} 

We now list some of the properties of the pair $(\mu_1(\cdot),v_1(\cdot))$ solution of the Cauchy problem \eqref{eq:InitialStep_Cauchy}. First, we have as a consequence of \eqref{eq:GeneralisedMomentum} in Remark \ref{rmk:GeneralisedMomentum} applied with $p= 1$ that the curve $\mu_1(\cdot)$ satisfies the momentum estimate 
\begin{equation*}
\M_1(\mu_1(t)) \leq \left( \M_1(\mu^0) + \INTSeg{m(s) \Big( 1 + \M_1(\nu(s)) \Big)}{s}{0}{t} \right) \exp \left(\INTSeg{m(s)}{s}{0}{t} \right). 
\end{equation*}
We now set $f_0(t) := \M_1(\nu(t))$ and $f_1(t) := \M_1(\mu_1(t))$ for all $t \in [0,T]$, and let $C > 0$ be as in Lemma \ref{lem:IntegralIneq} applied with $\alpha = (1+\M_1(\mu^0) + \Norm{m(\cdot)}_1) \exp (\Norm{m(\cdot)}_1)$ and $N = 2$. Then, it holds\footnote{Notice that $C$ depends only on the magnitude of $r,R_{\nu},T$ and $\Norm{m(\cdot)}_1$ and not on the curve $\nu(\cdot)$ itself.}
\begin{equation}
\label{eq:Uniform_MomentumBound}
\max\left\{ \NormC{\M_1(\nu(\cdot))}{0}{[0,T]} ~,~ \NormC{\M_1(\mu_1(\cdot))}{0}{[0,T]} \right\} \leq C.
\end{equation}
This implies in particular that $v_1(\cdot,\cdot)$ satisfies the uniform sub-linearity estimate 
\begin{equation*}
|v_1(t,x)| \leq (1+C)m(t) \big( 1 + |x| \big),
\end{equation*}
for $\Lcal^1$-almost every $t \in [0,T]$ and every $x \in \R^d$. Thus by \eqref{eq:Uniform_BoundLip} in Theorem \ref{thm:CauchyLipschitz}, there exist $R_r >0$ and $m_r(\cdot) \in L^1([0,T],\R_+)$ depending on the magnitude of $r,R_{\nu},T,C$ and $\Norm{m(\cdot)}_1$ such that  
\begin{equation}
\label{eq:Uniform_SupportBound}
\supp(\nu(t)) \cup \supp(\mu_1(t)) \subseteq K := B(0,R_r) \qquad \text{and} \qquad W_1(\mu_1(t),\mu_1(s)) \leq \INTSeg{m_r(\tau)}{\tau}{s}{t},
\end{equation}
for all times $0 \leq s \leq t \leq T$. By applying the estimate \eqref{eq:WassEstimate} of Proposition \ref{prop:MomentumGronwall} with $p =1$, we also recover 
\begin{equation*}
\begin{aligned}
W_1(\mu_1(t),\nu(t)) &\leq \left( W_1(\mu^0,\nu(0)) + \INTSeg{\NormC{v_1(s,\cdot) - \hat{w}(s,\cdot)}{0}{K,\R^d}}{s}{0}{t} \right) \exp \left(\INTSeg{l_K(s)}{s}{0}{t} \right) \\
& = \chi_1(t), 
\end{aligned}
\end{equation*}
for all $t \in [0,T]$, were we used \eqref{eq:Ineq_eta}, the expression of the map $\chi_1(\cdot)$ given in \eqref{eq:ChiCkp_Def}, and the fact that  
\begin{equation}
\label{eq:Extended_eta}
\NormC{v_1(t,\cdot) - \hat{w}(t,\cdot)}{0}{K,\R^d} = \hspace{0.1cm} \Norm{v_1(t,\cdot) - w_{\vert K_{\nu}}(t,\cdot)}_{C^0(K_{\nu},\R^d)} \hspace{0.1cm} = \hspace{0.1cm} \eta_{\nu}(t),
\end{equation}
for $\Lcal^1$-almost every $t \in [0,T]$.


\paragraph*{Step 2: Building the whole sequence $(\mu_n(\cdot),v_n(\cdot))$.}

By the support inclusion of \eqref{eq:Uniform_SupportBound} together with hypothesis \textbf{(DI)}-$(iv)$, there exists a map $L_K(\cdot) \in L^1([0,T],\R_+)$ such that
\begin{equation}
\label{eq:LipschitzInclusion_Wasserstein}
V_K(t,\mu_1(t)) \subset V_K(t,\nu(t)) + L_K(t) W_1(\mu_1(t),\nu(t)) \B_{C^0(K,\R^d)}, 
\end{equation}
for $\Lcal^1$-almost every $t \in [0,T]$. This along with hypothesis \textbf{(DI)}-$(i)$ and an application of Lemma \ref{lem:Measurable}-(a) and (b) yields the existence of a measurable selection $t \in [0,T] \mapsto v_2(t) \in V_K(t,\mu_1(t))$ such that 
\begin{equation*}
\NormC{v_2(t,\cdot)-v_1(t,\cdot)}{0}{K,\R^d} \leq L_K(t) W_1(\mu_1(t),\nu(t)), 
\end{equation*}
for $\Lcal^1$-almost every $t \in [0,T]$. As a consequence of the hypothesis \textbf{(DI)}-$(ii)$, \textbf{(DI)}-$(iii)$ and \eqref{eq:Uniform_MomentumBound}, one further has
\begin{equation*}
|v_2(t,x)| \leq (1+C)m(t) \big( 1 + |x| \big) \qquad \text{and} \qquad \Lip(v_2(t,\cdot) \, ; K) \leq l_K(t), 
\end{equation*}
for $\Lcal^1$-almost every $t \in [0,T]$ and any $x \in K$. By repeating the same extension argument as in Step 1 with $\pi_{K}(\cdot)$, we have that $v_2(\cdot,\cdot)$ satisfies the assumptions \textbf{(C1)}-\textbf{(C2)} of Theorem \ref{thm:CauchyLipschitz}. Whence, there exists a unique solution $\mu_2(\cdot) \in \AC([0,T],\Pcal_c(\R^d))$ of the Cauchy problem \eqref{eq:InitialStep_Cauchy} driven by $v_2(\cdot,\cdot)$.

By a direct application of \eqref{eq:GeneralisedMomentum} in Remark \ref{rmk:GeneralisedMomentum} and \eqref{eq:WassEstimate} of Proposition \ref{prop:MomentumGronwall}, we furthermore have that the curve $\mu_2(\cdot)$ satisfies the momentum and distance estimates 
\begin{equation*}
\left\{
\begin{aligned}
& \M_1(\mu_2(t)) \leq \left( \M_1(\mu^0) + \INTSeg{m(s) \Big( 1 + \M_1(\mu_1(s)) \Big)}{s}{0}{t} \right) \exp \left( \INTSeg{m(s)}{s}{0}{t} \right), \\
& W_1(\mu_2(t),\mu_1(t)) \leq \left( \INTSeg{\NormC{v_2(s,\cdot)-v_1(s,\cdot)}{0}{K,\R^d}}{s}{0}{t} \right) \exp \left( \INTSeg{l_K(s)}{s}{0}{t} \right), 
\end{aligned}
\right.
\end{equation*}
for all times $t \in [0,T]$. Thus by Lemma \ref{lem:IntegralIneq} applied with $N = 3$ and the constant $\alpha >0$ defined as before, the curve $\mu_2(\cdot)$ also satisfies the uniform momentum bound 
\begin{equation*}
\NormC{\M_1(\mu_2(\cdot))}{0}{[0,T]} \leq C.
\end{equation*}
By repeating this process, one can build a sequence of pairs $(\mu_n(\cdot),v_n(\cdot)) \subset \AC([0,T],\Pcal(K)) \times L^1([0,T],C^0(K,\R^d))$ satisfying \eqref{eq:Construction1}-\eqref{eq:Construction2} along with the uniform estimates
\begin{equation}
\label{eq:UniformEstProof}
\supp(\mu_n(t)) \subset K := B(0,R_r) \qquad \text{and} \qquad W_1(\mu_n(t),\mu_n(s)) \leq \INTSeg{m_r(\tau)}{\tau}{s}{t},
\end{equation}
for all times $0 \leq s  \leq t \leq T$ and any $n \geq 1$.


\paragraph*{Step 3: Convergence of the sequence $(\mu_n(\cdot),v_n(\cdot))$.}

The next step in our argument is to show that the sequence $(\mu_n(\cdot),v_n(\cdot))$ is a Cauchy sequence in the complete metric space $C^0([0,T],\Pcal(K)) \times L^1([0,T],C^0(K,\R^d)$. By applying \eqref{eq:WassEstimate}, we have for all $t \in [0,T]$ and any $n \geq 1$
\begin{equation}
\label{eq:Induction_DistEst}
\begin{aligned}
& \hspace{0.5cm} W_1(\mu_{n+1}(t),\mu_n(t)) \\
& \leq \left( \INTSeg{\NormC{v_{n+1}(s_n,\cdot) - v_n(s_n,\cdot)}{0}{K,\R^d}}{s_n}{0}{t} \right) \exp \left( \INTSeg{l_K(s)}{s}{0}{t} \right) \\
& \leq \left( \INTSeg{L_K(s_n) W_1 \big( \mu_n(s_n),\mu_{n-1}(s_n) \big)}{s_n}{0}{t} \right) \exp \left( \INTSeg{l_K(s)}{s}{0}{t} \right) \\
& \leq \left( \INTSeg{L_K(s_n) \INTSeg{ \NormC{v_n(s_{n-1},\cdot) - v_{n-1}(s_{n-1},\cdot)}{0}{K,\R^d} }{s_{n-1}}{0}{s_n}}{s_n}{0}{t} \right) \exp \left( 2 \INTSeg{l_K(s)}{s}{0}{t} \right) \\
& \hspace{0.25cm} \vdots \\
& \leq \Bigg( \INTSeg{L_K(s_n) \INTSeg{L_K(s_{n-1}) \dots \INTSeg{L_K(s_1) W_1 \big(\mu_1(s_1),\nu(s_1) \big) }{s_1}{0}{s_2}\dots}{s_{n-1}}{0}{s_n}}{s_n}{0}{t} \Bigg) \exp \left( n \INTSeg{l_K(s)}{s}{0}{t} \right) \\
& \leq \frac{\chi_1(t)}{n!}\left( \INTSeg{L_K(s)}{s}{0}{t} \right)^n  \exp \left( n \INTSeg{l_K(s)}{s}{0}{t} \right) \\
& = \frac{\chi_1(t)}{n!} C_{K,1}(t)^n, 
\end{aligned}
\end{equation}
where we recall that $\chi_1(\cdot)$ and $C_{K,1}(\cdot)$ are as in \eqref{eq:ChiCkp_Def} with $p=1$. Whence for any $m,n \geq 1$, it holds
\begin{equation*}
\begin{aligned}
\sup_{t \in [0,T]} W_1(\mu_{n+m}(t),\mu_n(t)) &  \leq \sup_{t \in [0,T]} \sum_{k=n}^{n+m-1} W_1(\mu_{k+1}(t),\mu_k(t)) \\
& \leq \sup_{t \in [0,T]} \bigg[ \chi_1(t) \sum_{k=n}^{n+m-1} \frac{C_{K,1}(t)^k}{k!} \bigg] ~\underset{m,n \rightarrow +\infty}{\longrightarrow}~ 0. 
\end{aligned}
\end{equation*}
Therefore, $(\mu_n(\cdot))$ is a Cauchy sequence in the complete metric space $C^0([0,T],\Pcal(K))$, and it converges uniformly towards a limit curve $\mu(\cdot) \in C^0([0,T],\Pcal(K))$. 

Similarly, we can show that $(v_n(\cdot))$ is a Cauchy sequence in $L^1([0,T],C^0(K,\R^d))$. Indeed for any $m,n \geq 1$, we have by using the second line of \eqref{eq:Construction2} that  
\begin{equation*}
\begin{aligned}
\INTSeg{\NormC{v_{n+m}(t,\cdot) - v_n(t,\cdot)}{0}{K,\R^d}}{t}{0}{T} & \leq \sum_{k=n}^{m+n-1} \INTSeg{\NormC{v_{k+1}(t,\cdot) - v_k(t,\cdot)}{0}{K,\R^d}}{t}{0}{T} \\
& \leq \sum_{k=n}^{m+n-1} \INTSeg{L_K(t) W_1(\mu_{k}(t),\mu_{k-1}(t))}{t}{0}{T}  \\
& \leq \Norm{L_K(\cdot)}_1 \sup_{t \in [0,T]} \bigg[ \chi_1(t)  \sum_{k=n-1}^{m+n-2} \frac{C_{K,1}(t)^k}{k!} \bigg] ~\underset{m,n \rightarrow +\infty}{\longrightarrow}~ 0.
\end{aligned}
\end{equation*}
Whence, the sequence $(v_n(\cdot)) \subset L^1([0,T],C^0(K,\R^d))$ is Cauchy and converges towards a limit map $v(\cdot) \in L^1([0,T],C^0(K,\R^d))$. Moreover, one can easily prove as a consequence of hypotheses \textbf{(DI)}-$(ii)$ and \textbf{(DI)}-$(iii)$ that the limit velocity field satisfies the estimates 
\begin{equation}
\label{eq:Estimates_Limitvel}
|v(t,x)| \leq (1+C)m(t) \big( 1 + |x| \big) \qquad \text{and} \qquad \Lip(v(t,\cdot) \, ; K) \leq l_K(t),
\end{equation}
for $\Lcal^1$-almost every $t \in [0,T]$ and every $x \in K$. 


\paragraph*{Step 4 : Properties of the limit pair $(\mu(\cdot),v(\cdot))$.}

As a consequence of classical stability results for continuity equations (see e.g. \cite{MFOC} or Section \ref{section:Existence} below), the limit pair $(\mu(\cdot),v(\cdot))$ is a distributional solution of the Cauchy problem
\begin{equation}
\label{eq:CauchyProblem}
\left\{
\begin{aligned}
& \partial_t \mu(t) + \Div (v(t)\mu(t)) = 0, \\
& \mu(0) = \mu^0.
\end{aligned} 
\right.
\end{equation}
Moreover by \eqref{eq:Estimates_Limitvel}, the velocity field $v : [0,T] \times \R^d \rightarrow \R^d$ satisfies up to an extension argument the hypotheses \textbf{(C1)}-\textbf{(C2)} of Theorem \ref{thm:CauchyLipschitz}. Thus, the curve of measures $\mu(\cdot) \in \AC([0,T],\Pcal_c(\R^d))$ is the unique solution of \eqref{eq:CauchyProblem}. By taking the limit as $n \rightarrow +\infty$ in \eqref{eq:UniformEstProof}, we further obtain that 
\begin{equation*}
\supp(\mu(t)) \subset K := B(0,R_r) \qquad \text{and} \qquad W_1(\mu(t),\mu(s)) \leq \INTSeg{m_r(\tau)}{\tau}{s}{t},
\end{equation*}
for all times $0 \leq s \leq t \leq T$.

We now want to prove that $(\mu(\cdot),v(\cdot))$ is a trajectory-selection pair of the differential inclusion \eqref{eq:Theorem_DiffInc}. We start by observing that the first line of \eqref{eq:Construction2} can be restated as
\begin{equation}
\label{eq:GraphInclusion}
\Big( \mu_n(t),v_{n+1}(t) \Big) \in \Vcal_K(t)
\end{equation}
for $\Lcal^1$-almost every $t \in [0,T]$ and any $n \geq 1$. Here, the sets $\Vcal_K(t)$ are the graphs of the set-valued maps $\mu \in \Pcal(K) \rightrightarrows V_K(t,\mu)$, defined for $\Lcal^1$-almost every $t \in [0,T]$ by 
\begin{equation*}
\Vcal_K(t) := \Big\{ (\mu,v) \in \Pcal(K) \times C^0(K,\R^d) ~\text{s.t.}~ v \in V_K(t,\mu) \Big\}. 
\end{equation*}
Observe that it follows from hypothesis \textbf{(DI)}-$(iv)$ and elementary properties of set-valued maps that the multifunction $\mu \in \Pcal(K) \rightrightarrows V_K(t,\mu)$ is closed, so that the sets $\Vcal_K(t)$ are closed in the $W_1 \times C^0(K,\R^d)$-topology for $\Lcal^1$-almost every $t \in [0,T]$. As a consequence of Step 3, we further have that
\begin{equation*}
W_1(\mu(t),\mu_n(t)) ~\underset{n \rightarrow +\infty}{\longrightarrow}~ 0,
\end{equation*}
for all times $t \in [0,T]$, as well as 
\begin{equation*}
\NormC{v(t,\cdot)-v_n(t,\cdot)}{0}{K,\R^d}  ~\underset{n \rightarrow +\infty}{\longrightarrow}~ 0,
\end{equation*}
for $\Lcal^1$-almost every $t \in [0,T]$ along an adequate subsequence that we do not relabel. Therefore, taking the limit as $n \rightarrow +\infty$ in \eqref{eq:GraphInclusion}, we recover for $\Lcal^1$-almost every $t \in [0,T]$ the inclusion 
\begin{equation}
\label{eq:FinalVelocityInc}
v(t) \in V_K(t,\mu(t)).
\end{equation}
Observe now that \eqref{eq:CauchyProblem} and \eqref{eq:FinalVelocityInc} together imply that $(\mu(\cdot),v(\cdot))$ is a trajectory-selection pair of the differential inclusion \eqref{eq:Theorem_DiffInc} in the sense of Definition \ref{def:WassInc}. 

We end the proof of Theorem \ref{thm:FilippovWass} by deriving the velocity and trajectory estimates \eqref{eq:Theorem_DistEst}-\eqref{eq:Theorem_VelEst}. First, remark that for any $n \geq 1$ and all times $t \in [0,T]$, it holds 
\begin{equation*}
W_1(\mu_n(t),\nu(t)) \leq \sum_{k=0}^{n-1} W_1(\mu_{k+1}(t),\mu_k(t))  \leq \chi_1(t) \sum_{k=0}^{n-1} \frac{C_{K,1}(t)^k}{k!} \, \leq \chi_1(t) \,\exp \big( C_{K,1}(t) \big), 
\end{equation*}
where we used \eqref{eq:Induction_DistEst}. Taking the limit as $n \rightarrow +\infty$ in the previous expression, we recover \eqref{eq:Theorem_DistEst}. Concerning the velocity estimate, we have for any $n \geq 1$ that 
\begin{equation*}
\begin{aligned}
\NormC{v_n(t,\cdot) - \hat{w}(t,\cdot)}{0}{K,\R^d} & \leq \sum_{k=1}^{n-1} \NormC{v_{k+1}(t,\cdot) - v_k(t,\cdot)}{0}{K,\R^d} + \NormC{v_1(t,\cdot) - \hat{w}(t,\cdot)}{0}{K,\R^d} \\
& \leq L_K(t) \chi_1(t) \exp \big( C_{K,1}(t) \big) + \eta_{\nu}(t),
\end{aligned}
\end{equation*}
for $\Lcal^1$-almost every $t \in [0,T]$, where we used the second line in \eqref{eq:Construction2}, \eqref{eq:Extended_eta} and \eqref{eq:Induction_DistEst}. Taking again the limit as $n \rightarrow +\infty$ along a suitable subsequence in the previous expression, we recover \eqref{eq:Theorem_VelEst}.
\end{proof}


\subsection{Relaxation Theorem}
\label{subsection:Relax}

In this section, we state and prove a measure theoretic pendant of the Relaxation Theorem in the Wasserstein space $(\Pcal_c(\R^d),W_p)$.

\begin{thm}
\label{thm:RelaxationWass}
Let $T > 0$, $p \in [1,+\infty)$, $\mu^0 \in \Pcal_c(\R^d)$, $V : [0,T] \times \Pcal_c(\R^d) \rightrightarrows C^0(\R^d,\R^d)$ be a set-valued map satisfying \textbn{(DI)}, and $\mu(\cdot) \in \AC([0,T],\Pcal_c(\R^d))$ be a solution of the \textnormal{relaxed differential inclusion}
\begin{equation}
\label{eq:RelaxedInc_Theorem}
\left\{
\begin{aligned}
& \partial_t \mu(t) \in - \Div \Big( \co V(t,\mu(t)) \mu(t) \Big), \\
& \mu(0) = \mu^0,
\end{aligned}
\right.
\end{equation}
where $\co \big( V(t,\mu(t)) \big)$ is the closed convex hull of $V(t,\mu(t)) \subset C^0(\R^d,\R^d)$ defined in the sense of \eqref{eq:ConvbarDef}. 

Then for any $\delta > 0$, there exists a solution $\mu_{\delta}(\cdot) \in \AC([0,T],\Pcal_c(\R^d))$ of the differential inclusion 
\begin{equation}
\label{eq:Inc_Theorem}
\left\{
\begin{aligned}
& \partial_t \mu_{\delta}(t) \in - \Div \Big( V(t,\mu_{\delta}(t)) \mu_{\delta}(t) \Big), \\
& \mu_{\delta}(0) = \mu^0, 
\end{aligned}
\right.
\end{equation}
such that 
\begin{equation*}
W_p(\mu(t),\mu_{\delta}(t)) \leq \delta, 
\end{equation*}
for all times $t \in [0,T]$.
\end{thm}

The proof of Theorem \ref{thm:RelaxationWass} is split into two steps. Step 1 consists in building an auxiliary measure curve $\nu(\cdot)$ driven by a measurable selection from $t \in [0,T] \mapsto w(t) \in V(t,\mu(t))$ which is close to the solution $\mu(\cdot)$ of \eqref{eq:RelaxedInc_Theorem}. In Step 2, we apply the Filippov estimates of Theorem \ref{thm:FilippovWass} to recover the existence of a solution $\mu_{\delta}(\cdot)$ of \eqref{eq:Inc_Theorem} which is close to $\nu(\cdot)$ and therefore close to $\mu(\cdot)$ as well.

\begin{proof}[Proof (of Theorem \ref{thm:RelaxationWass})]
As we did in the proof of Theorem \ref{thm:FilippovWass}, we will restrict our attention to the case $p=1$, the general case being similar. Let us denote by $r > 0$ a positive radius such that $\mu^0 \in \Pcal(B(0,r))$. By Proposition \ref{prop:IncSuppBound}, there exists a compact set $K := B(0,R_r) \subset \R^d$ depending on $r,T$ and $\Norm{m(\cdot)}_1$ such that all the solutions of \eqref{eq:RelaxedInc_Theorem} and \eqref{eq:Inc_Theorem} starting from $\mu^0$ are uniformly compactly supported in $K$. This together with hypothesis \textbf{(DI)}-$(ii)$ implies in particular that every measurable selection $t \in [0,T] \mapsto v(t) \in C^0(\R^d,\R^d)$ of either $V(t,\mu(t))$ or $\co V(t,\mu(t))$ is such that 
\begin{equation}
\label{eq:SublinearityRelax}
\NormC{v(t,\cdot)}{0}{K,\R^d} \leq m(t) \big( 1 + 2 R_r \big),
\end{equation}
for $\Lcal^1$-almost every $t \in [0,T]$. 


\paragraph*{Step 1: Construction of an intermediate curve $\nu(\cdot)$.}

Since $m(\cdot)\in L^1([0,T],\R_+)$, we can find for any given $\delta' >0$ a subdivision of $[0,T]$ into $N \geq 1$ intervals $[t_i,t_{i+1}]$ such that
\begin{equation}
\label{eq:Cutting1}
\sup_{i \in \{0,\dots,N-1\}} \, \INTSeg{ \hspace{-0.15cm} m(s)}{s}{t_i}{t_{i+1}} \leq \frac{\delta'}{2(1+2R_r)}, 
\end{equation}
where $N$ depends on $\delta'$. Moreover, remark that the multifunction  $t \in [0,T] \rightrightarrows V_K(t,\mu(t))$ is $\Lcal^1$-measurable and integrably bounded with compact values in $C^0(K,\R^d)$. Thus by Lemma \ref{lem:Aumann}, there exists a family of measurable selections $t \in [t_i,t_{i+1}] \mapsto v_i(t) \in V_K(t,\mu(t))$ such that
\begin{equation}
\label{eq:Cutting2}
\NormC{ \, \INTSeg{v(s,\cdot)}{s}{t_i}{t_{i+1}} - \INTSeg{v_i(s,\cdot)}{s}{t_i}{t_{i+1}} \, }{0}{K,\R^d} \leq \frac{\delta'}{N},
\end{equation}
for all $i \in \{0,\dots,N-1\}$. Let us now consider the following Carathéodory velocity field 
\begin{equation*}
w : (t,x) \in [0,T] \times K \mapsto \sum_{i = 0}^{N-1} \mathds{1}_{[t_i,t_{i+1}]}(t) v_i(t,x) \in \R^d,
\end{equation*}
which satisfies hypotheses \textbf{(C1)}-\textbf{(C2)} of Theorem \ref{thm:CauchyLipschitz} up to an extension, and denote by $\nu(\cdot) \in \AC([0,T],$ $\Pcal(K))$ the unique solution of the Cauchy problem
\begin{equation*}
\left\{
\begin{aligned}
& \partial_t \nu(t) + \Div(w(t)\nu(t)) = 0, \\
& \nu(0) = \mu^0.
\end{aligned}
\right.
\end{equation*}
Our next goal is to estimate the $W_1$-distance between $\mu(t)$ and $\nu(t)$ for all times $t \in [0,T]$.

Let $\Beta_{\mu},\Beta_{\nu}$ be superposition measures which induce $\mu(\cdot)$ and $
\nu(\cdot)$ as in Theorem \ref{thm:Superposition}. By Lemma \ref{lem:OptimalMeasures}, there exists a transport plan $\hat{\Beta}_{\mu,\nu} \in \Gamma(\Beta_{\mu},\Beta_{\nu})$ such that \eqref{eq:OptimalPlan_Statement} holds, i.e.
\begin{equation*}
(\pi_{\R^d},\pi_{\R^d})_{\#} \hat{\Beta}_{\mu,\nu} \in \Gamma_o(\mu^0,\mu^0) \qquad \text{and} \qquad (e_t,e_t)_{\#} \hat{\Beta}_{\mu,\nu} \in \Gamma_o(\mu(t),\nu(t)),
\end{equation*}
for all times $t \in [0,T]$. Thus, one has 
\begin{equation}
\label{eq:Relaxation_Estimate}
\begin{aligned}
W_1(\mu(t),\nu(t)) & = \INTDom{|x-y|}{\R^{2d}}{\Big( \big(e_t,e_t \big)_{\#} \hat{\Beta}_{\mu,\nu} \Big)(x,y)} \\
& = \INTDom{|\sigma_{\mu}(t) - \sigma_{\nu}(t)|}{(\R^d \times \Sigma_T)^2}{\hat{\Beta}_{\mu,\nu}(x,\sigma_{\mu},y,\sigma_{\nu})} \\
& \leq \INTDom{\left|\INTSeg{\Big( v(s,\sigma_{\mu}(s)) - v(s,\sigma_{\nu}(s))\Big)}{s}{0}{t} \right|}{(\R^d \times \Sigma_T)^2}{\hat{\Beta}_{\mu,\nu}(x,\sigma_{\mu},y,\sigma_{\nu})} \\
& \hspace{0.45cm} + \INTDom{\left|\INTSeg{\Big( v(s,\sigma_{\nu}(s)) - w(s,\sigma_{\nu}(s))\Big)}{s}{0}{t} \right|}{(\R^d \times \Sigma_T)^2}{\hat{\Beta}_{\mu,\nu}(x,\sigma_{\mu},y,\sigma_{\nu})},
\end{aligned}
\end{equation}
where we used that $\Beta_{\mu},\Beta_{\nu}$ are concentrated on the characteristic curves \eqref{eq:Characteristic_Def} of $v(\cdot,\cdot)$ and $w(\cdot,\cdot)$, along with the known fact that $\Gamma_o(\mu^0,\mu^0) = \{ (\Id,\Id)_{\#} \mu^0 \}$. The first integral in \eqref{eq:Relaxation_Estimate} can be estimated\footnote{When  $p >1$, this estimate can be performed as in the proof of Proposition \ref{prop:MomentumGronwall} given in Appendix \ref{appendix:Proofs}.} as 
\begin{equation}
\label{eq:IntegralEstimate1}
\begin{aligned}
& \INTDom{\left|\INTSeg{\Big( v(s,\sigma_{\mu}(s)) - v(s,\sigma_{\nu}(s))\Big)}{s}{0}{t} \right|}{(\R^d \times \Sigma_T)^2}{\hat{\Beta}_{\mu,\nu}(x,\sigma_{\mu},y,\sigma_{\nu})} \\
& \hspace{1.2cm} \leq \INTSeg{l_K(s) \INTDom{|\sigma_{\mu}(s) - \sigma_{\nu}(s)|}{(\R^d \times \Sigma_T)^2}{\hat{\Beta}_{\mu,\nu}(x,\sigma_{\mu},y,\sigma_{\nu})}}{s}{0}{t} = \INTSeg{l_K(s)W_1(\mu(s),\nu(s))}{s}{0}{t},
\end{aligned}
\end{equation}
where we used Fubini's Theorem and the fact that $x \in K \mapsto v(s,x)$ is $l_K(s)$-Lipschitz for $\Lcal^1$-almost every $s \in [0,t]$ by hypothesis \textbf{(DI)}-$(iii)$. Let $j \in \{1,\dots,N-1\}$ be such that $t \in [t_j,t_{j+1}]$. For the second integral term in \eqref{eq:Relaxation_Estimate}, we then have
\begin{equation}
\label{eq:IntegralEstimate2}
\begin{aligned}
 & \hspace{0.475cm} \INTDom{\left|\INTSeg{\Big( v(s,\sigma_{\nu}(s)) - w(s,\sigma_{\nu}(s))\Big)}{s}{0}{t} \right|}{(\R^d \times \Sigma_T)^2}{\hat{\Beta}_{\mu,\nu}(x,\sigma_{\mu},y,\sigma_{\nu})} \\
& \leq \INTDom{\left| \sum_{i=0}^{j-1} \INTSeg{\Big( v(s,\sigma_{\nu}(s)) - v_i(s,\sigma_{\nu}(s))\Big)}{s}{t_i}{t_{i+1}} \right|}{(\R^d \times \Sigma_T)^2}{\hat{\Beta}_{\mu,\nu}(x,\sigma_{\mu},y,\sigma_{\nu})} \\
 & \hspace{0.45cm} + \INTDom{\left| \INTSeg{\Big( v(s,\sigma_{\nu}(s)) - v_j(s,\sigma_{\nu}(s))\Big)}{s}{t_j}{t} \right|}{(\R^d \times \Sigma_T)^2}{\hat{\Beta}_{\mu,\nu}(x,\sigma_{\mu},y,\sigma_{\nu})} \\
& \leq \sum_{i=0}^{j-1} \INTDom{\left| \INTSeg{\Big( v(s,\sigma_{\nu}(s)) - v_i(s,\sigma_{\nu}(s))\Big)}{s}{t_i}{t_{i+1}} \right|}{(\R^d \times \Sigma_T)^2}{\hat{\Beta}_{\mu,\nu}(x,\sigma_{\mu},y,\sigma_{\nu})} + \delta',
\end{aligned}
\end{equation}
since $\NormC{v(s,\cdot) - v_j(s,\cdot)}{0}{K,\R^d} \leq 2 (1+2 R_r)m(s)$ for $\Lcal^1$-almost every $s \in [t_j,t_{j+1}]$. For any $i \in \{0,\dots,j-1\}$, it further holds 
\begin{equation}
\label{eq:IntegralEstimate3}
\begin{aligned}
& \hspace{0.475cm} \INTDom{\left| \INTSeg{\Big( v(s,\sigma_{\nu}(s)) - v_i(s,\sigma_{\nu}(s))\Big)}{s}{t_i}{t_{i+1}} \right|}{(\R^d \times \Sigma_T)^2}{\hat{\Beta}_{\mu,\nu}(x,\sigma_{\mu},y,\sigma_{\nu})} \\
& \leq \INTDom{\left| \INTSeg{\Big( v(s,\sigma_{\nu}(t_i)) - v_i(s,\sigma_{\nu}(t_i))\Big)}{s}{t_i}{t_{i+1}} \right|}{(\R^d \times \Sigma_T)^2}{\hat{\Beta}_{\mu,\nu}(x,\sigma_{\mu},y,\sigma_{\nu})}  \\
& \hspace{0.45cm} + \INTDom{\left| \INTSeg{\Big( v(s,\sigma_{\nu}(s)) - v(s,\sigma_{\nu}(t_i))\Big)}{s}{t_i}{t_{i+1}} \right|}{(\R^d \times \Sigma_T)^2}{\hat{\Beta}_{\mu,\nu}(x,\sigma_{\mu},y,\sigma_{\nu})}   \\
& \hspace{0.45cm} + \INTDom{\left| \INTSeg{\Big( v_i(s,\sigma_{\nu}(s)) - v_i(s,\sigma_{\nu}(t_i))\Big)}{s}{t_i}{t_{i+1}} \right|}{(\R^d \times \Sigma_T)^2}{\hat{\Beta}_{\mu,\nu}(x,\sigma_{\mu},y,\sigma_{\nu})} \\
& \leq \frac{\delta'}{N} + 2 \INTSeg{l_K(s) \INTDom{|\sigma_{\nu}(s) - \sigma_{\nu}(t_i) |}{(\R^d \times \Sigma_T)^2}{\hat{\Beta}_{\mu,\nu}(x,\sigma_{\mu},y,\sigma_{\nu})}}{s}{t_i}{t_{i+1}} \\
& \leq \frac{\delta'}{N} + 2 (1+2R_r) \INTSeg{l_K(s) \left( \INTSeg{m(\tau)}{\tau}{t_i}{s} \right)}{s}{t_i}{t_{i+1}} \\
& \leq \frac{\delta'}{N} + \delta' \INTSeg{l_K(s)}{s}{t_i}{t_{i+1}}
\end{aligned}
\end{equation}
where we used \eqref{eq:SublinearityRelax}, \eqref{eq:Cutting1} and \eqref{eq:Cutting2}. 

Plugging \eqref{eq:IntegralEstimate1}, \eqref{eq:IntegralEstimate2} and \eqref{eq:IntegralEstimate3} into \eqref{eq:Relaxation_Estimate}, we have shown that the intermediate curve $\nu(\cdot)$ is such that
\begin{equation*}
W_1(\mu(t),\nu(t)) \leq \INTSeg{l_K(s) W_1(\mu(s),\nu(s))}{s}{0}{t} + \delta' \left( 2 + \INTSeg{l_K(s)}{s}{0}{t} \right),
\end{equation*}
for all times $t \in [0,T]$. By a direct application of Gr\"onwall's Lemma, we finally obtain 
\begin{equation}
\label{eq:Distance_Estimate1}
\sup_{t \in [0,T]} W_1(\mu(t),\nu(t)) \leq \delta' \, \Big( 2 \, + \Norm{l_K(\cdot)}_1 \Big) \exp \left( \Norm{l_K(\cdot)}_1 \right).
\end{equation}


\paragraph*{Step 2: Construction of the curve $\mu_{\delta}(\cdot)$ solution of \eqref{eq:Inc_Theorem}.}

Observe that the intermediate curve $\nu(\cdot) \in \AC([0,T],\Pcal(K))$ that we have built in Step 1 is not a solution of \eqref{eq:Inc_Theorem}, since by construction its driving velocity field $w : [0,T] \times \R^d \rightarrow \R^d$ is such that 
\begin{equation*}
w(t) \in V_K(t,\mu(t)),
\end{equation*}
for $\Lcal^1$-almost every $t \in [0,T]$. We introduce the mismatch function $\eta_{\nu}(\cdot)$ defined in this context by
\begin{equation*}
\eta_{\nu} : t \in [0,T] \mapsto \dist_{C^0(K,\R^d)} \Big( w(t),V_K(t,\nu(t)) \Big),
\end{equation*}
and notice that as a consequence of \textbf{(DI)}-$(iv)$, there exists $L_K(\cdot) \in L^1([0,T],\R_+)$ such that 
\begin{equation}
\label{eq:Distance_Estimate1bis}
\eta_{\nu}(t) \leq L_K(t)W_1(\mu(t),\nu(t)),
\end{equation}
for $\Lcal^1$-almost every $t \in [0,T]$. In particular, it follows from \eqref{eq:Distance_Estimate1} that $\eta_{\nu}(\cdot) \in L^1([0,T],\R_+)$. 

We can therefore apply Theorem \ref{thm:FilippovWass} to deduce the existence of a solution $\mu_{\delta}(\cdot) \in \AC([0,T],\Pcal(K))$ of the differential inclusion \eqref{eq:Inc_Theorem} such that 
\begin{equation}
\label{eq:Distance_Estimate2}
W_1(\nu(t),\mu_{\delta}(t)) \leq \left( \INTSeg{\eta_{\nu}(s)}{s}{0}{t} \right) \exp \left(C_{K,1}(t) + \NormL{l_K(\cdot)}{1}{[0,t]} \right),
\end{equation}
for all times $t \in [0,T]$, as a consequence of \eqref{eq:Theorem_DistEst}. By combining \eqref{eq:Distance_Estimate1}, \eqref{eq:Distance_Estimate1bis} and \eqref{eq:Distance_Estimate2} along with an application of the triangle inequality, we further have that 
\begin{equation*}
\begin{aligned}
\sup_{t \in [0,T]} W_1(\mu(t),\mu_{\delta}(t)) \leq ~ & \bigg( 1 +  \Norm{L_K(\cdot)}_1 \exp \Big( C_{K,1}(T) + \Norm{l_K(\cdot)}_1 \hspace{-0.1cm} \Big) \bigg) \sup_{t \in [0,T]} W_1(\mu(t),\nu(t)) \\
\leq ~ & \delta' \bigg( 1 + \Norm{L_K(\cdot)}_1 \exp \Big(C_{K,1}(T) + \Norm{l_K(\cdot)}_1 \hspace{-0.1cm} \Big) \bigg) \Big( 2 \, + \Norm{l_K(\cdot)}_1 \hspace{-0.1cm} \Big) \exp \left( \Norm{l_K(\cdot)}_1 \right).
\end{aligned}
\end{equation*}
Thus, choosing 
\begin{equation*}
\delta' = \frac{\delta}{\bigg( 1 + \Norm{L_K(\cdot)}_1 \exp \Big(C_{K,1}(T) + \Norm{l_K(\cdot)}_1 \hspace{-0.1cm} \Big) \bigg) \Big( 2 \, + \Norm{l_K(\cdot)}_1 \hspace{-0.1cm}  \Big) \exp \left( \Norm{l_K(\cdot)}_1 \right)},
\end{equation*}
we obtain the uniform distance estimate 
\begin{equation*}
\sup_{t \in [0,T]} W_1(\mu(t),\mu_{\delta}(t)) \leq \delta, 
\end{equation*}
which concludes the proof of Theorem \ref{thm:RelaxationWass}.
\end{proof}

We now apply Theorem \ref{thm:RelaxationWass} to recover a classical fact, which states that the value-function corresponding to a minimisation problem with a convexified right-hand side coincides with the value function of the original problem, and that both value functions are continuous in a certain sense. 

\begin{Def}[Locally continuous maps]
\label{def:LocC0}
We say that a functional $\phi : \Pcal_c(\R^d) \rightarrow \R$ is \textnormal{locally continuous} if for any compact set $K \subset \R^d$ and all $\mu \in \Pcal(K)$, it holds 
\begin{equation*}
\phi(\mu_n) ~\underset{n \rightarrow +\infty}{\longrightarrow}~ \phi(\mu), 
\end{equation*}
for every sequence $(\mu_n) \subset \Pcal(K)$ such that $\mu_n \rightharpoonup^* \mu$ as $n \rightarrow +\infty$. 
\end{Def}

\begin{cor}[Local continuity of the value function]
Let $V : [0,T] \times \Pcal_c(\R^d) \rightrightarrows C^0(\R^d,\R^d)$ be a set-valued map satisfying \textbn{(DI)} for some $p \in [1,+\infty)$ and $\varphi : \Pcal_c(\R^d) \rightarrow \R$ be a locally continuous map. Then, the \textnormal{value functions} $\Vcal,\Vcal_{\co} : [0,T] \times\Pcal_c(\R^d) \rightarrow \R$ defined respectively by 
\begin{equation*}
\Vcal \big( \tau , \mu_{\tau} \big) :=~ \left\{ 
\begin{aligned}
\inf_{\mu(\cdot)} \, & \big[ \varphi(\mu(T)) \big] \\
\textnormal{s.t.}~ & 
\left\{ 
\begin{aligned}
& \partial_t \mu(t) \in - \Div \Big( V(t,\mu(t)) \mu(t) \Big), \\
& \mu(\tau) = \mu_{\tau}, 
\end{aligned}
\right.
\end{aligned}
\right.
\end{equation*}
and
\begin{equation*}
~~ \Vcal_{\co} \big( \tau , \mu_{\tau} \big) :=~ \left\{ 
\begin{aligned}
\inf_{\mu(\cdot)} \, & \big[ \varphi(\mu(T)) \big] \\
\textnormal{s.t.}~ & 
\left\{ 
\begin{aligned}
& \partial_t \mu(t) \in - \Div \Big( \co V(t,\mu(t)) \mu(t) \Big), \\
& \mu(\tau) = \mu_{\tau}, 
\end{aligned}
\right.
\end{aligned}
\right.
\end{equation*}
for all $(\tau,\mu_{\tau}) \in [0,T] \times \Pcal_c(\R^d)$ are equal and locally continuous.
\end{cor}

\begin{proof}
Let $(\tau,\mu_{\tau}) \in [0,T] \times \Pcal_c(\R^d)$ and $r_{\tau} > 0$ be such that $\mu_{\tau} \in \Pcal(B(0,r_{\tau}))$. By Proposition \ref{prop:IncSuppBound}, there exists a compact set $K_{\tau} := B(0,R_{r_{\tau}})$ in which all the trajectories of the differential inclusions \eqref{eq:RelaxedInc_Theorem} and \eqref{eq:Inc_Theorem} starting from $\mu_{\tau}$ at time $\tau \in [0,T]$ are uniformly compactly supported. 

Since every trajectory of \eqref{eq:Inc_Theorem} is also a trajectory of the relaxed inclusion \eqref{eq:RelaxedInc_Theorem}, it directly holds 
\begin{equation}
\label{eq:ValueFunction}
\Vcal_{\co}(\tau,\mu_{\tau}) \leq \Vcal(\tau,\mu_{\tau}). 
\end{equation}
Conversely, let $\mu^*(\cdot)$ be a trajectory of the relaxed inclusion \eqref{eq:RelaxedInc_Theorem}. By Theorem \ref{thm:RelaxationWass}, there exists a sequence $(\mu_n(\cdot)) \subset \AC([0,T],\Pcal(K_{\tau}))$ of trajectories of \eqref{eq:Inc_Theorem} such that 
\begin{equation*}
\sup_{t \in [\tau,T]} W_p(\mu_n(t),\mu^*(t)) ~\underset{n \rightarrow +\infty}{\longrightarrow}~ 0.
\end{equation*}
Since the trajectories $\{\mu_n(\cdot),\mu^*(\cdot)\}$ are uniformly supported in $K_{\tau}$, this implies by Proposition \ref{prop:Properties_Wp} that $\mu_n(T) \rightharpoonup^* \mu^*(T)$ as $n \rightarrow +\infty$. Recalling that $\varphi(\cdot)$ is locally continuous in the sense of Definition \ref{def:LocC0}, we deduce that for every $\epsilon > 0$ there exists an integer $N_{\epsilon} \geq 1$ -- which depends on $\mu^*(\cdot)$ --, such that 
\begin{equation}
\label{eq:EpsilonIneq}
\Vcal(\tau,\mu_{\tau}) \leq \varphi(\mu_n(T)) \leq \varphi(\mu^*(T)) + \epsilon,
\end{equation}
for every $n \geq N_{\epsilon}$. Thus taking the infimum over the trajectories $\mu^*(\cdot)$ in \eqref{eq:EpsilonIneq}, we recover 
\begin{equation*}
\Vcal(\tau,\mu_{\tau}) \leq \Vcal_{\co}(\tau,\mu_{\tau}) + \epsilon, 
\end{equation*}
for every $\epsilon > 0$, which together with \eqref{eq:ValueFunction} yields that $\Vcal(\tau,\mu_{\tau}) = \Vcal_{\co}(\tau,\mu_{\tau})$. 

The continuity of the map $t \in [0,T] \mapsto \Vcal(t,\mu_{\tau})$ and the local continuity of $\mu \in \Pcal_c(\R^d) \mapsto \Vcal(\tau,\mu)$ in the sense of Definition \ref{def:LocC0} at $(\tau,\mu_{\tau})$ follow easily from the estimates \eqref{eq:Prop_SuppLip} of Proposition \ref{prop:IncSuppBound} and \eqref{eq:Theorem_DistEst} of Theorem \ref{thm:FilippovWass} together with the local continuity of $\varphi(\cdot)$.
\end{proof}


\subsection{Compactness of the set of trajectories}
\label{subsection:Compactness}

In this section, we show that the set of solutions to a differential inclusion in $(\Pcal_c(\R^d),W_p)$ is compact in the topology of the uniform convergence whenever its right-hand side has convex values.

\begin{thm}[Compactness of trajectories]
\label{thm:Compactness}
Let $\mu^0 \in \Pcal_c(\R^d)$ and $V : [0,T] \times \Pcal_c(\R^d) \rightrightarrows C^0(\R^d,\R^d)$ be a set-valued map with convex values which satisfies hypotheses \textbn{(DI)} for some $p \in [1,+\infty)$. Then, the solution set of the differential inclusion \eqref{eq:DefWassInc} with $\mu(0) = \mu^0$ is compact in $C^0([0,T],\Pcal_c(\R^d))$. 
\end{thm}

The proof of Theorem \ref{thm:Compactness} strongly relies on by-now classical estimates and compactness results for continuity equations, for which we refer the reader e.g. to \cite{FPR,MFOC}

\begin{proof}[Proof (of Theorem \ref{thm:Compactness})]
Let $r >0$ be such that $\mu^0 \in \Pcal(B(0,r))$ and $(\mu_n(\cdot),v_n(\cdot))$ be a sequence of trajectory-selection pairs for \eqref{eq:DefWassInc} with $\mu_n(0) = \mu^0$ for every $n \geq 1$. By Proposition \ref{prop:IncSuppBound}, there exist a constant $R_r >0$ and a map $m_r(\cdot) \in L^1([0,T],\R_+)$ such that 
\begin{equation}
\label{eq:Proof_UnifSuppLip}
\supp(\mu_n(t)) \subset K := B(0,R_r) \qquad \text{and} \qquad W_p(\mu_n(t),\mu_n(s)) \leq \INTSeg{m_r(\tau)}{\tau}{s}{t},
\end{equation}
for all $0 \leq s \leq t \leq T$ and every $n \geq 1$. Therefore by the Ascoli-Arzelà Theorem, their exists a subsequence of $(\mu_n(\cdot))$ that we do not relabel and a limit curve $\mu^*(\cdot)$ such that 
\begin{equation}
\label{eq:UnifConvergenceMeasures}
\sup_{t \in [0,T]} W_p(\mu_n(t),\mu^*(t)) ~\underset{n \rightarrow +\infty}{\longrightarrow}~0.
\end{equation}
It can moreover be verified straightforwardly  that the limit curve $\mu^*(\cdot)$ satisfies the estimates of \eqref{eq:Proof_UnifSuppLip}. 

As a consequence of hypotheses \textbf{(DI)}-$(ii)$ and \textbf{(DI)}-$(iii)$ along with the uniform compactness of the support of the trajectories $\{(\mu_n(\cdot)),\mu^*(\cdot)\}$ given by \eqref{eq:Proof_UnifSuppLip}, the sequence of admissible velocities $(v_n(\cdot)) \subset L^1([0,T],C^0(K,\R^d))$ is uniformly integrably bounded in $L^1([0,T],W^{1,q}(K,\R^d))$ for any $q \in (1,+\infty)$. By an application of the generalisation of Dunford-Pettis Theorem (see e.g. \cite[Theorem 1.38]{AmbrosioFuscoPallara}) for Bochner integrable maps, the sequence $(v_n(\cdot))$ admits a cluster point $v^*(\cdot)$ in the weak $L^1([0,T],W^{1,q}(K,\R^d))$-topology, i.e.
\begin{equation}
\label{eq:DualityPairing}
\INTSeg{ \hspace{-0.1cm} \big\langle \Bnu(t) , v^*(t) - v_n(t) \big\rangle_{W^{1,q}(K,\R^d)} }{t}{0}{T} ~\underset{n \rightarrow +\infty}{\longrightarrow}~0,
\end{equation}
for any $\Bnu(\cdot) \in L^{\infty}([0,T],(W^{1,q}(K,\R^d))')$. We henceforth choose an exponent $q > d$, so that by Morrey's Embedding (see e.g. \cite[Theorem 9.12]{Brezis}) it holds that $W^{1,q}(K,\R^d) \subset C^0(K,\R^d)$. By taking the topological dual of this inclusion, we recover that $(C^0(K,\R^d))' \subset (W^{1,q}(K,\R^d))'$, so that \eqref{eq:DualityPairing} implies in particular 
\begin{equation*}
\INTSeg{\INTDom{\langle \nabla_x \phi(t,x) , v^*(t,x) - v_n(t,x) \rangle}{\R^d}{\mu^*(t)(x)}}{t}{0}{T} ~\underset{n \rightarrow +\infty}{\longrightarrow}~0,
\end{equation*}
for any $\phi \in C^{\infty}_c([0,T] \times \R^d)$. Recalling that the maps $v_n(t,\cdot)$ are $l_K(t)$-Lipschitz for any $n \geq 1$ and $\Lcal^1$-almost every $t \in [0,T]$, we obtain
\begin{equation*}
\INTSeg{\INTDom{\langle \nabla_x \phi(t,x) , v_n(t,x) \rangle}{\R^d}{\mu_n(t)(x)}}{t}{0}{T} ~\underset{n \rightarrow +\infty}{\longrightarrow}~ \INTSeg{\INTDom{\langle \nabla_x \phi(t,x) , v^*(t,x) \rangle}{\R^d}{\mu^*(t)(x)}}{t}{0}{T},
\end{equation*} 
for any $\phi \in C^{\infty}_c([0,T] \times \R^d)$, where we used the Kantorovich duality formula \eqref{eq:Kantorovich_duality} and the fact that the Wasserstein distances are ordered, see Proposition \ref{prop:Properties_Wp}. Passing to the limit as $n \rightarrow +\infty$ in the distributional formulation \eqref{eq:ContinuityDistrib1} of the continuity equation, we conclude that the pair $(\mu^*(\cdot),v^*(\cdot))$ is a solution of the Cauchy problem 
\begin{equation*}
\left\{
\begin{aligned}
& \partial_t \mu^*(t) + \Div \big( v^*(t)\mu^*(t) \big) = 0,\\
& \mu^*(0) = \mu^0. 
\end{aligned}
\right.
\end{equation*}

We now prove that $v^*(t) \in V_K(t,\mu^*(t))$ for $\Lcal^1$-almost every $t \in [0,T]$. As a consequence of hypothesis \textbf{(DI)}-$(iv)$ along with an application of Lemma \ref{lem:Measurable}-(a), there exists a sequence of maps $(\tilde{v}_n(\cdot)) \subset L^1([0,T],C^0(K,\R^d))$ such that 
\begin{equation}
\label{eq:ProofLip}
\NormC{v_n(t,\cdot) - \tilde{v}_n(t,\cdot)}{0}{K,\R^d} \leq L_K(t) \, W_p(\mu_n(t),\mu^*(t)) \qquad \text{and} \qquad \tilde{v}_n(t) \in V_K(t,\mu^*(t)), \\
\end{equation}
for $\Lcal^1$-almost every $t \in [0,T]$. By repeating the same compactness argument as before, the sequence of maps $(\tilde{v}_n(\cdot))$ admits a cluster point $\tilde{v}(\cdot)$ in the weak $L^1([0,T],W^{1,q}(K,\R^d))$-topology. It can in turn be checked as a consequence of our standing assumptions that the set
\begin{equation}
\label{eq:Proof_Unif}
\Vcal_K := \bigg\{ v \in L^1([0,T],C^0(K,\R^d)) ~\text{s.t.}~ v(t) \in V_K(t,\mu^*(t))  ~\text{for $\Lcal^1$-almost every $t \in [0,T]$} \bigg\},
\end{equation}
is closed in the strong $L^1([0,T],W^{1,q}(K,\R^d))$-topology for any $q \in (d,+\infty)$. Observe that it is also convex since the set-valued map $V(\cdot,\cdot)$ has convex values. Hence it is also weakly closed by Mazur's Lemma (see e.g. \cite[Theorem 3.7]{Brezis}), which implies that $\tilde{v}(\cdot) \in \Vcal_K$. Now we can conclude that $v^*(\cdot) \in \Vcal_K$ as a consequence of \eqref{eq:UnifConvergenceMeasures} and \eqref{eq:ProofLip} and by uniqueness of the weak-$^*$ limit.
\end{proof}


\section{Application to a Mean-Field Optimal Control Problem}
\label{section:Existence}

In this section, we apply the set-theoretic tools and results of Section \ref{section:WassDiff} to study the existence of minimisers for a constrained mean-field optimal control problem. To the best of our knowledge, this is the first existence result of this type for general constrained problems, with the extra novelty that the controls may have a feedback structure and are involved non-linearly in the dynamics. 

Consider $\varphi : \Pcal_c(\R^d) \rightarrow \R$ and $\Qpazo_T,\K \subset \Pcal_1(\R^d)$. Moreover, let $(U,d_U)$ be a compact metric space, $v : [0,T] \times \Pcal_c(\R^d) \times U\rightarrow C^0(\R^d,\R^d)$ be a non-local velocity field and $\mu^0 \in \Pcal_c(\R^d)$. In the sequel, we shall study the following general constrained Mayer problem
\begin{equation*}
(\Ppazo) ~ 
\left\{
\begin{aligned}
\min_{u(\cdot) \in \U} & \, \big[ \varphi(\mu(T)) \big] \\
\text{s.t.} ~ & \left\{
\begin{aligned}
& \partial_t \mu(t) + \Div \Big( v( t,\mu(t), u(t,\cdot)) \mu(t) \Big) = 0, \\
& \mu(0) = \mu^0, 
\end{aligned}
\right. \\
\text{and} ~ & \left\{
\begin{aligned}
\mu(T) & \in \Qpazo_T, \\
\mu(t) \, & \in \K ~~ \text{for all times $t \in [0,T]$},
\end{aligned}
\right.
\end{aligned}
\right.
\end{equation*}
over the set of admissible controls $u : [0,T] \rightarrow \Omega$, where $\Omega \subset C^0(\R^d,U)$ is a closed set such that
\begin{equation}
\label{eq:LUDef}
 \Omega \subset \Big\{ \omega \in C^0(\R^d,U) ~\text{s.t.}~ \Lip(\omega(\cdot) \, ; \R^d) \leq L_U \Big\}, 
\end{equation}
for a given constant $L_U > 0$. Here, the control variables $ u : (t,x) \in [0,T] \times \R^d \rightarrow U$ are Carathéodory vector fields such that $u(t,\cdot) \in \Omega$ for $\Lcal^1$-almost every $t \in [0,T]$, which pilot the evolution of the state $\mu(\cdot)$ of the system through the non-local controlled velocity
\begin{equation*}
(t,x) \in [0,T] \times \R^d \mapsto v \big( t, \mu(t),u(t,x) \big)(x) \in \R^d. 
\end{equation*}
We denote by $\U$ the set of all such admissible controls. 

\begin{rmk}[The case of open-loop controls]
Observe that the above definition of admissible controls is also adapted to the study of purely open-loop controls, i.e. controls which depend on time only and not on the space variable $x \in \R^d$. Indeed, let $\Omega = \{ \omega \in C^0(\R^d,U) ~\text{s.t.}~ \omega(\cdot) \equiv \text{const} \}$. Then $\Omega$ is closed and for any $\Lcal^1$-measurable selection $t \in [0,T] \mapsto u(t) \in \Omega$, the mapping $\tilde{u} : (t,x) \in [0,T] \times \R^d \mapsto u(t) \in U$ satisfies $\tilde{u}(t,\cdot) \in \Omega$ for $\Lcal^1$-almost every $t \in [0,T]$. Setting 
\begin{equation*}
\U := \Big\{ \tilde{u}(\cdot,\cdot) ~\text{s.t.}~ \tilde{u}(t,\cdot) \in \Omega ~\text{and $\tilde{u}(\cdot,x)$ is $\Lcal^1$-measurable}  \Big\},
\end{equation*}
we get the set of admissible purely open-loop controls. 
\end{rmk}

We consider the following differential inclusion 
\begin{equation}
\label{eq:DiffIncOCP}
\partial_t \mu(t) \in - \Div \Big( V(t,\mu(t)) \mu(t) \Big),
\end{equation}
where the set-valued map $V : [0,T] \times \Pcal_c(\R^d) \rightrightarrows C^0(\R^d,\R^d)$ is defined by 
\begin{equation}
\label{eq:SetValuedControlDyn}
V(t,\mu) := \Big\{ \vb \in C^0(\R^d,\R^d) ~\text{s.t.}~ \vb(\cdot) = \hat{v} \big( t,\mu,\omega\big)(\cdot)~ \text{for some $\omega \in \Omega$} \, \Big\},
\end{equation}
and the map $\hat{v}(t,\mu,\omega) \in C^0(\R^d,\R^d)$ is given for any $\omega \in \Omega$ by 
\begin{equation}
\label{eq:Def_vhat}
\hat{v}(t,\mu,\omega) : x \in \R^d \mapsto v \big( t,\mu,\omega(x)\big)(x). 
\end{equation}
Throughout this section, we fix $p \in [1,+\infty)$ and impose the following assumptions on problem $(\Ppazo)$. 

\begin{hyp}[\textbf{OCP}]
For every $R > 0$, assume that the following holds with $K := B(0,R)$.
\begin{enumerate}
\item[$(i)$] The map $t \in [0,T] \mapsto v(t,\mu,u)(x) \in \R^d$ is $\Lcal^1$-measurable and there exists $m(\cdot) \in L^1([0,T],\R_+)$ such that 
\begin{equation*}
\big| v(t,\mu,u)(x) \big| \leq m(t) \Big( 1 + |x| + \M_1(\mu) \Big),
\end{equation*}
for $\Lcal^1$-almost every $t \in [0,T]$ and any $(\mu,u,x) \in \Pcal_c(\R^d) \times U \times \R^d$. Moreover, there exist two maps $l_K(\cdot),L_K(\cdot) \in L^1([0,T],\R_+)$ such that for $\Lcal^1$-almost every $t \in [0,T]$, we have
\begin{equation*}
\big| v(t,\mu,u_1)(x) - v(t,\mu,u_2)(y) \big| \leq l_K(t) \Big( |x-y| + d_U(u_1,u_2)\Big),
\end{equation*}
for any $x,y \in K$ and $u_1,u_2 \in U$, and 
\begin{equation*}
\big\| v (t,\mu,u)(\cdot) - v(t,\nu,u)(\cdot) \big\|_{C^0(K,\R^d)} \leq L_K(t) W_p(\mu,\nu), 
\end{equation*}
for any $\mu,\nu \in \Pcal(K)$ and $u \in U$.  
\item[$(ii)$] The set of admissible velocities $V(t,\mu)$ defined in \eqref{eq:SetValuedControlDyn} is convex for $\Lcal^1$-almost every $t \in [0,T]$ and every $\mu \in \Pcal_c(\R^d)$.
\item[$(iii)$] The final cost $\mu \in \Pcal_c(\R^d) \mapsto \varphi(\mu) \in \R $ is lower-semicontinuous over $\Pcal(K)$ in the $W_1$-metric. 
\item[$(iv)$] The running and final constraint sets $\K$ and $\Qpazo_T$ are closed in the $W_1$-topology.
\end{enumerate}
\end{hyp}

\begin{rmk}[Open-loop controls and regularity]
When $\U$ consists of open-loop controls only, the Lipschitz continuity assumption on $u \in U \mapsto v(t,\mu,u) \in \R^d$ can be relaxed into a continuity assumption. 
\end{rmk}

We refer the reader e.g. to \cite[Appendix A]{PMPWassConst}, \cite[Section4]{PMPWass}, as well as \cite{Cavagnari2020} and \cite{FLOS,FPR} for examples of velocity fields, cost functionals and constraint sets satisfying similar assumptions. In the following proposition, we show that under hypotheses \textbn{(OCP)}-$(i)$, the set of all trajectories of the controlled non-local Cauchy-problem
\begin{equation}
\label{eq:ControlledCauchy}
\left\{
\begin{aligned}
& \partial_t \mu(t) + \Div \Big( v( t,\mu(t), u(t,\cdot)) \mu(t) \Big) = 0, \\
& \mu(0) = \mu^0 \in \Pcal_c(\R^d), 
\end{aligned}
\right.
\end{equation}
coincides with the set of trajectories of \eqref{eq:DiffIncOCP} whenever $V(\cdot,\cdot)$ is defined by \eqref{eq:SetValuedControlDyn}.

\begin{prop}[Link between differential inclusions and control systems]
\label{prop:DiffInc_Control}
Let $v : [0,T] \times \Pcal_c(\R^d) \times \R^d \times U \rightarrow \R^d$ be a non-local velocity field satisfying hypothesis \textbn{(OCP)}-$(i)$ and $V : [0,T] \times \Pcal_c(\R^d) \rightrightarrows C^0(\R^d,\R^d)$ be the set-valued map defined as in \eqref{eq:SetValuedControlDyn}. 

Then, a curve of measures $\mu(\cdot) \in \AC([0,T],\Pcal_c(\R^d))$ is a solution of the differential inclusion \eqref{eq:DiffIncOCP} if and only if it is a solution of \eqref{eq:ControlledCauchy} generated by an admissible control $u(\cdot) \in \U$.
\end{prop}

\begin{proof}
Let $\mu(\cdot) \in \AC([0,T],\Pcal_c(\R^d))$ be a curve of measures solution of \eqref{eq:ControlledCauchy} for a given admissible control map $u(\cdot) \in \U$. By construction, the time-dependent velocity field $\vb : t \in [0,T] \mapsto \hat{v}(t,\mu(t),u(t)) \in C^0(\R^d,\R^d)$ is such that $\vb(t) \in V(t,\mu(t))$ for $\Lcal^1$-almost every $t \in [0,T]$, so that $\mu(\cdot)$ solves \eqref{eq:DiffIncOCP}.

Conversely, suppose that $\mu(\cdot) \in \AC([0,T],\Pcal_c(\R^d))$ is a solution of \eqref{eq:DiffIncOCP}. Notice that as a consequence of hypothesis \textbn{(OCP)}-$(i)$, and of the definition of $\U$, the set valued map $V(\cdot,\cdot)$ defined in \eqref{eq:SetValuedControlDyn} satisfies hypotheses \textbn{(DI)}. Whence by Proposition \ref{prop:IncSuppBound}, there exists a closed ball $K := B(0,R)$ such that $\supp(\mu(t)) \subset K$ for all times $t \in [0,T]$, and by Definition \ref{def:WassInc} we obtain the existence of a measurable selection $t \in [0,T] \mapsto \vb(t) \in V_K(t,\mu(t))$ such that 
\begin{equation*}
\partial_t \mu(t) + \Div \big( \vb(t) \mu(t) \big) = 0. 
\end{equation*}
Moreover, we know that $V_K(t,\mu(t)) = \hat{v}_{|K}(t,\mu(t),\Omega) \subset C^0(K,\R^d)$ where $(t,\omega) \mapsto \hat{v}_{|K}(t,\mu(t),\omega)$ is $\Lcal^1$-measurable with respect to $t \in [0,T]$ and continuous with respect to $\omega \in \Omega$. We can therefore apply the measurable selection theorem e.g. of \cite[Theorem 8.2.9]{Aubin1990} to recover the existence of a measurable selection $t \in [0,T] \mapsto u(t) \in \Omega$ such that $\vb(t) = \hat{v}_{|K}(t,\mu(t),u(t))$ for $\Lcal^1$-almost every $t \in [0,T]$. Therefore up to an extension argument, we deduce that $\mu(\cdot)$ solves \eqref{eq:ControlledCauchy} with driving velocity field $(t,x) \in [0,T] \times \R^d \mapsto v(t,\mu(t),u(t,x))(x) \in \R^d$.  
\end{proof}

\begin{rmk}[Comparison with the admissible trajectories of \cite{Cavagnari2018,CavagnariMP2018,Jimenez2020}]
\label{rmk:AdmissibleTraj}
In \cite{Cavagnari2018,CavagnariMP2018,Jimenez2020}, the authors consider a different notion of solution to differential inclusion in Wasserstein spaces. Given a compact metric space $U$ and a continuous map $f : \Pcal_2(\R^d) \times \R^d \times U \rightarrow \R^d$ which is Lipschitz with respect to its two first arguments, they define the set-valued map $F(\mu,x) := f(\mu,x,U)$. A trajectory $\mu(\cdot)$ is then said to be admissible if it solves \eqref{eq:ContinuityEquation} driven by a Borel velocity field $(t,x) \mapsto v(t,x) \in F(\mu(t),x)$. It is shown that for this notion of solution to differential inclusion, admissible trajectories depend in a Lipschitz-like way on their initial condition. The proof of this result relies on a careful adaptation of the superposition principle from \cite{Ambrosio2014} (see also \cite{CavagnariMP2018}), that allows to link the set of characteristics of the  differential inclusion $\dot \sigma (t) \in F(\mu(t),\sigma(t))$ to the curve of measures $\mu(\cdot)$ via the evaluation map.

In this context however, the controls $u_{\sigma}(\cdot)$ obtained by applying measurable selection theorems to $t \in [0,T] \rightrightarrows f(\mu(t),\sigma(t),U)$ inherently depend on both the measure curve $\mu(\cdot)$ and the characteristic curve $\sigma(\cdot)$. This is a crucial difference with our definition of admissible trajectories, for which admissible controls depend on the state $\mu(\cdot)$ only. For this reason, our functional approach to differential inclusions is closer in spirit to the usual formulation of control systems as differential inclusions. Besides as illustrated in the Introduction, it is also more meaningful in terms of the geometry of the metric spaces $(\Pcal_c(\R^d),W_p)$ seen as subsets of the pseudo-Riemannian manifold $(\Pcal_2(\R^d),W_2)$.
\end{rmk}

In Theorem \ref{thm:Existence} below, we state a general result on the existence of optimal controls for problem $(\Ppazo)$. We would like to stress that in most of the existing contributions on this topic (see e.g. \cite{PMPWassConst,PMPWass,FLOS,MFOC}), the velocity field $(t,x,\mu,\omega) \mapsto v(t,\mu,\omega(x))(x)$ is assumed to have a \textit{control-affine} structure. The case of non-linearly controlled vector-field was studied e.g. in \cite{Burger2019,Cavagnari2018,Pogodaev2016} for open-loop controls. 

\begin{thm}[Existence of optimal controls for $(\Ppazo)$]
\label{thm:Existence}
Under hypotheses \textbn{(OCP)}, there exists an optimal trajectory-control pair $(\mu^*(\cdot),u^*(\cdot)) \in \AC([0,T],\Pcal_c(\R^d)) \times \U$ for $(\Ppazo)$.
\end{thm}

\begin{proof}
Let $(u_n(\cdot)) \subset \U$ be a minimising sequence for $(\Ppazo)$ and $(\mu_n(\cdot)) \subset \AC([0,T],\Pcal_c(\R^d))$ be the corresponding sequence of solutions of the non-local Cauchy problems
\begin{equation}
\label{eq:NonLocal_ContinuityEq}
\left\{ 
\begin{aligned}
& \partial_t \mu_n(t) + \Div \Big( v\big( t,\mu_n(t),u_n(t,\cdot) \big) \mu_n(t) \Big) = 0, \\
& \mu_n(0) = \mu^0. 
\end{aligned}
\right.
\end{equation}
It can be checked that the set-valued map $V(\cdot,\cdot)$ defined in \eqref{eq:SetValuedControlDyn} satisfies the set of assumptions \textbf{(DI)} as a consequence of hypotheses \textbf{(OCP)}-$(i)$ together with the definition of $\U$. Thus by Proposition \ref{prop:IncSuppBound}, there exists a compact set $K \subset \R^d$ such that $\supp(\mu_n(t)) \subset K$ for all $t \in [0,T]$ and any $n \geq 1$. Moreover, the admissible velocity sets $V_K(\cdot,\cdot)$ have convex values by \textbf{(OCP)}-$(ii)$. Whence by Theorem \ref{thm:Compactness}, there  exists a trajectory-selection pair $(\mu^*(\cdot),\vb^*(\cdot)) \in \AC([0,T],\Pcal(K)) \times L^1([0,T],C^0(K,\R^d))$ solution of the differential inclusion 
\begin{equation*}
\left\{
\begin{aligned}
& \partial \mu^*(t) \in - \Div \Big( V(t,\mu^*(t)) \mu^*(t) \Big), \\
& \mu^*(0) = \mu^0, 
\end{aligned}
\right.
\end{equation*}
such that 
\begin{equation*}
\sup_{t \in [0,T]} W_p(\mu_n(t),\mu^*(t)) ~\underset{n \rightarrow +\infty}{\longrightarrow}~ 0,
\end{equation*}
along a subsequence that we do not relabel. From Proposition \ref{prop:DiffInc_Control}, we deduce the existence of a measurable selection $t \in [0,T] \mapsto u^*(t) \in \Omega$ such that 
\begin{equation*}
\vb^*(t) = \hat{v}(t,\mu^*(t),u^*(t)), 
\end{equation*}
for $\Lcal^1$-almost every $t \in [0,T]$. By \eqref{eq:Def_vhat}, this implies that the limit trajectory-control pair $(\mu^*(\cdot),u^*(\cdot))$ is a solution of the Cauchy problem driving $(\Ppazo)$.

We can now conclude that $(\mu^*(\cdot),u^*(\cdot))$ is optimal for $(\Ppazo)$ by remarking that 
\begin{equation*}
\liminf_{n \rightarrow +\infty} \big[ \varphi(\mu_n(T)) \big] \geq \varphi(\mu^*(T)),
\end{equation*}
and also that $\mu^*(T) \in \Qpazo_T$ as well as $\mu^*(t) \in \K$ for all times $t \in [0,T]$, as a direct consequence of \textbf{(OCP)}-$(iii)$ and \textbf{(OCP)}-$(iv)$ along with \eqref{eq:UnifConvergenceMeasures}.
\end{proof}


\appendix 
\setcounter{thm}{0} \renewcommand{\thethm}{A.\arabic{thm}} 
\setcounter{prop}{0} \renewcommand{\theprop}{A.\arabic{prop}} 
\setcounter{equation}{0} \renewcommand{\theequation}{A.\arabic{equation}} 

\section{Proofs of Lemma \ref{lem:OptimalMeasures} and Proposition \ref{prop:MomentumGronwall}}
\label{appendix:Proofs}

In this section, we detail the proofs of Lemma \ref{lem:OptimalMeasures} and Proposition \ref{prop:MomentumGronwall}. These results rely strongly on the notion of \textit{disintegration} of measures defined over Banach spaces, which we recall in the following theorem (see e.g. \cite[Theorem 5.3.1]{AGS}). 

\begin{thm}[Disintegration]
\label{thm:Disintegration}
Let $X,Y$ be two separable Banach spaces and $\pi : Y \rightarrow X$ be a Borel map. Given a measure $\nu \in \Pcal(Y)$ and its image $\mu = \pi_{\#} \nu \in \Pcal(X)$ through $\pi$, there exists a $\mu$-almost uniquely determined family of Borel measures $\{ \nu_x \}_{x \in X} \subset \Pcal(Y)$ such that 
\begin{equation}
\label{eq:Disintegration}
\left\{
\begin{aligned}
\INTDom{\phi(y)}{Y}{\nu(y)} & = \INTDom{\left( \INTDom{\phi(y)}{\pi^{-1}(x)}{\nu_x(y)} \right)}{X}{\mu(x)}, \\
\nu_x(Y \backslash \pi^{-1}(x)) & = 0 \hspace{0.68cm} \text{for $\mu$-almost every $x \in X$},
\end{aligned}
\right.
\end{equation}
for any map $\phi \in L^1(Y,\R;\nu)$. The family of measures $\{ \nu_x \}_{x \in X}$ is called the \textnormal{disintegration of $\nu$ onto $\mu$} and is denoted by $\nu = \INTDom{\nu_x}{X}{\mu(x)}$.
\end{thm}


\begin{proof}[Proof (of Lemma \ref{lem:OptimalMeasures})]
Let $\{ t_k \}_{k = 1}^{+\infty} \subset [0,T]$ be a countable and dense subset in $[0,T]$. We are going to split the proof of this result into three steps. In Step 1, we start by building a suitable sequence of measures $(\hat{\Beta}^n_{\mu,\nu}) \subset \Gamma(\Beta_{\mu},\Beta_{\nu})$ such that $(e_{t_k},e_{t_k})_{\#} \hat{\Beta}^n_{\mu,\nu} \in \Gamma_o(\mu(t_k),\nu(t_k))$ for any $k \in \{1,\dots,n\}$. We then show in Step 2 that this sequence is tight and therefore narrowly sequentially compact in $\Pcal((\R^d \times \Sigma_T)^2)$, and finally in Step 3 that its cluster points $\hat{\Beta}_{\mu,\nu}$ satisfy \eqref{eq:OptimalPlan_Statement}.


\paragraph*{Step 1: Construction of the sequence $(\hat{\Beta}_{\mu,\nu}^n)$.}

Let $\Beta_{\mu},\Beta_{\nu} \in \Pcal(\R^d \times \Sigma_T)$ be defined as in the statement of Lemma \ref{lem:OptimalMeasures}. Let also $n \geq 1$ be an arbitrary integer and $\xb = (x_0,x_1,\dots,x_n)$ denote a generic element of $\R^d \times (\R^d)^n$. We start by defining the measures $\Bmu_n,\Bnu_n \in \Pcal(\R^d \times (\R^d)^n)$ as
\begin{equation*}
\Bmu_n := \left( \pi_{\R^d} , \prod_{k=1}^n e_{t_k} \right)_{\raisebox{8pt}{$\scriptstyle \#$}} \Beta_{\mu} \in \Gamma \left( \mu^0, \prod_{k=1}^n \mu(t_k) \right), \qquad \Bnu_n := \left( \pi_{\R^d} , \prod_{k=1}^n e_{t_k} \right)_{\raisebox{8pt}{$\scriptstyle \#$}} \Beta_{\nu} \in \Gamma \left( \nu^0, \prod_{k=1}^n \nu(t_k) \right). 
\end{equation*} 
By Theorem \ref{thm:Disintegration} above, there respectively exist a $\Bmu_n$-almost uniquely determined family of measures $\{ \Beta_{\mu,\xb}^n \}_{\xb} \subset \Pcal(\R^d \times \Sigma_T)$ and a $\Bnu_n$-almost uniquely determined family of measures $\{ \Beta_{\nu,\yb}^n \}_{\yb} \subset \Pcal(\R^d \times \Sigma_T)$ such that 
\begin{equation*}
\Beta_{\mu} = \INTDom{\Beta_{\mu,\xb}^n}{\R^d \times (\R^d)^n}{\Bmu_n(\xb)} \qquad \text{and} \qquad \Beta_{\nu} = \INTDom{\Beta_{\nu,\yb}^n}{\R^d \times (\R^d)^n}{\Bnu_n(\yb)}. 
\end{equation*}

Given $p \in [1,+\infty)$, let us choose $p$-optimal transport plans $\gamma_0 \in \Gamma_o(\mu^0,\nu^0)$ and $\gamma_k \in \Gamma_o(\mu(t_k),\nu(t_k))$ for any $k \in \{1,\dots,n\}$. By iterative applications of the Gluing Lemma (see e.g. \cite[Lemma 5.3.2]{AGS}), we can build a transport plan $\hat{\Bgamma}_n \in \Gamma(\Bmu_n,\Bnu_n)$ such that 
\begin{equation*}
(\pi^1,\pi^{n+2})_{\#} \hat{\Bgamma}_n = \gamma_0, \qquad (\pi^{k+1},\pi^{n+k+2})_{\#} \hat{\Bgamma}_n = \gamma_k.
\end{equation*}
for every $k \in \{1,\dots,n\}$. We can thus build by disintegration the measure $\hat{\Beta}_{\mu,\nu}^n \in \Pcal((\R^d \times \Sigma_T)^2)$ as
\begin{equation*}
\hat{\Beta}_{\mu,\nu}^n := \INTDom{ \Big( \Beta_{\mu,\xb}^n \times \Beta_{\nu,\yb}^n \Big)}{\R^{2d} \times (\R^{2d})^n}{\hat{\Bgamma}_n(\xb,\yb)}. 
\end{equation*}
Remark that by construction, it holds 
\begin{equation*}
(\pi_{\R^d},\pi_{\R^d})_{\#} \hat{\Beta}_{\mu,\nu}^n \in \Gamma_o(\mu^0,\nu^0) \qquad \text{and} \qquad (e_{t_k},e_{t_k})_{\#} \hat{\Beta}_{\mu,\nu}^n \in \Gamma_o(\mu(t_k),\nu(t_k)),
\end{equation*}
for any $k \in \{1,\dots,n\}$, together with $\hat{\Beta}^n_{\mu,\nu} \in \Gamma(\Beta_{\mu},\Beta_{\nu})$ for any $n \geq 1$. 

\paragraph*{Step 2: Tightness.}

We now want to prove that the sequence of measures $(\hat{\Beta}_{\mu,\nu}^n) \subset \Pcal((\R^d \times \Sigma_T)^2)$ is relatively sequentially compact. In separable Banach spaces, this is equivalent to the \textit{tightness} of the sequence by Prokhorov's Theorem (see e.g. \cite[Theorem 5.1.3]{AGS}). A necessary and sufficient condition for tightness (see e.g. \cite[Remark 5.1.5]{AGS}) is given in our context by the existence of a map $\Psi : (\R^d \times \Sigma_T)^2 \rightarrow [0,+\infty]$ with \textit{compact sub-levels}, such that 
\begin{equation*}
\sup_{n \geq 1} \INTDom{\Psi (x,\sigma_{\mu},y,\sigma_{\nu})}{(\R^d \times \Sigma_T)^2}{\hat{\Beta}^n_{\mu,\nu}(x,\sigma_{\mu},y,\sigma_{\nu})} < +\infty. 
\end{equation*}
It has been shown e.g. in \cite[Theorem 3.4]{AmbrosioC2014} that the functional 
\begin{equation*}
\psi : (x,\sigma) \in \R^d \times \Sigma_T \mapsto \left\{
\begin{aligned}
& |x| + \INTSeg{\frac{|\dot \sigma(t)|}{1+|\sigma(t)|}}{t}{0}{T} ~~ & \text{if $\sigma \in \AC([0,T],\R^d)$,} \\ 
& + \infty ~~ & \text{otherwise,}
\end{aligned}
\right.
\end{equation*}
has compact sublevels in $\R^d \times \Sigma_T$. In addition, notice that 
\begin{equation*}
\begin{aligned}
\INTDom{\psi(x,\sigma_{\mu})}{\R^d \times \Sigma_T}{\Beta_{\mu}(x,\sigma_{\mu})} & = \INTDom{|x|}{\R^d}{\mu^0(x)} + \INTDom{\left( \INTSeg{\frac{|\dot \sigma_{\mu}(t)|}{1+|\sigma_{\mu}(t)|}}{t}{0}{T} \right)}{\R^d \times \Sigma_T}{\Beta_{\mu}(x,\sigma_{\mu})} \\
& = \INTDom{|x|}{\R^d}{\mu^0(x)} + \INTSeg{\INTDom{\frac{|v(t,\sigma_{\mu}(t))|}{1+|\sigma_{\mu}(t)|}}{\R^d \times \Sigma_T}{\Beta_{\mu}(x,\sigma_{\mu})}}{t}{0}{T} \\
& = \INTDom{|x|}{\R^d}{\mu^0(x)} + \INTSeg{\INTDom{\frac{|v(t,x)|}{1+|x|}}{\R^d}{\mu(t)(x)}}{t}{0}{T} \hspace{0.1cm} \leq \hspace{0.1cm} C, 
\end{aligned}
\end{equation*}
for a constant $C > 0$ depending only on $K$ and $\Norm{m(\cdot)}_1$, where we used \eqref{eq:Characteristic_Def}, Fubini's Theorem, and the sub-linearity estimate \textbn{(C1)}. The same estimate also holds true for $\Beta_{\nu}$, so that the map
\begin{equation*}
\Psi : (x,\sigma_{\mu},y,\sigma_{\nu}) \mapsto \psi(x,\sigma_{\mu}) + \psi(y,\sigma_{\nu}),
\end{equation*}
has compact sub-levels in $(\R^d \times \Sigma_T)^2$ and is such that
\begin{equation*}
\sup_{n \geq 1} \INTDom{\Psi(x,\sigma_{\mu},y,\sigma_{\nu})}{(\R^d \times \Sigma_T)^2}{\hat{\Beta}^n_{\mu,\nu}(x,\sigma_{\mu},y,\sigma_{\nu})} < +\infty.
\end{equation*}
Whence, the sequence $(\hat{\Beta}^n_{\mu,\nu})$ is tight in $\Pcal((\R^d \times \Sigma_T)^2)$ and therefore narrowly sequentially compact. 


\paragraph*{Step 3: Optimality of the cluster points.}

Let $\hat{\Beta}_{\mu,\nu}$ be a cluster point of $(\hat{\Beta}^n_{\mu,\nu})$ along a subsequence that we do not relabel in the narrow topology of $\Pcal((\R^d \times \Sigma_T)^2)$. By construction, it holds for any $k \in \{1,\dots,n\}$ that
\begin{equation}
\label{eq:OptimalityPlans}
(e_{t_k},e_{t_k})_{\#} \hat{\Beta}^n_{\mu,\nu} \in \Gamma_o(\mu(t_k),\nu(t_k)). 
\end{equation}
Let $t \in [0,T]$ be arbitrary and $(t_{k_m}) \subset (t_k)$ be a subsequence such that $t_{k_m} \rightarrow t$ as $m \rightarrow +\infty$. By the continuity of the evaluation maps $(e_{t_{k_m}}) \subset C^0(\R^d \times \Sigma_T,\R^d)$ for any $m \geq 1$ together with classical convergence results on pushforwards of sequence of measures (see \cite[Lemma 5.2.1]{AGS}), one has
\begin{equation*}
(e_{t_{k_m}},e_{t_{k_m}})_{\#} \hat{\Beta}^n_{\mu,\nu} ~\underset{n \rightarrow +\infty}{\rightharpoonup^*}~ (e_{t_{k_m}},e_{t_{k_m}})_{\#} \hat{\Beta}_{\mu,\nu}.
\end{equation*}
In addition, remark that for any $m \geq 1$, the integrals
\begin{equation*}
\INTDom{|x-y|^p }{\R^{2d}}{\Big( (e_{t_{k_m}},e_{t_{k_m}})_{\#} \hat{\Beta}^n_{\mu,\nu} \Big)(x,y)} = \INTDom{|\sigma_{\mu}(t_{k_m}) - \sigma_{\nu}(t_{k_m})|^p}{(\R^d \times \Sigma_T)^2}{\hat{\Beta}^n_{\mu,\nu}(x,\sigma_{\mu},y,\sigma_{\nu})},
\end{equation*}
are bounded uniformly with respect to $n \geq 1$, since $\mu(\cdot)$ and $\nu(\cdot)$ are uniformly supported in a compact set $K \subset \R^d$. By the stability under narrow convergence of $\Gamma_o(\mu(t_{k_m}),\nu(t_{k_m}))$ (see e.g. \cite[Proposition 7.1.3]{AGS}) together with \eqref{eq:OptimalityPlans}, this further implies 
\begin{equation*}
(e_{t_{k_m}},e_{t_{k_m}})_{\#} \hat{\Beta}_{\mu,\nu} \in \Gamma_o(\mu(t_{k_m}),\nu(t_{k_m})),
\end{equation*}
for every $m \geq 1$. 

We now let $m \rightarrow +\infty$. By the narrow continuity of the curves $\mu(\cdot)$ and $\nu(\cdot)$, we have
\begin{equation*}
\mu(t_{k_m}) ~\underset{m \rightarrow +\infty}{\rightharpoonup^*}~ \mu(t) \qquad \text{and} \qquad \nu(t_{k_m}) ~\underset{m \rightarrow +\infty}{\rightharpoonup^*}~ \nu(t). 
\end{equation*}
Moreover, observe that 
\begin{equation*}
\sup_{(x,\sigma_{\mu}) \in \supp (\Beta_{\mu})} \big| e_t(x,\sigma_{\mu}) - e_{t_{k_m}}(x,\sigma_{\mu}) \big| = \sup_{(x,\sigma_{\mu}) \in \supp (\Beta_{\mu})} \big| \sigma_{\mu}(t) - \sigma_{\mu}(t_{k_m}) \big| ~\underset{m \rightarrow +\infty}{\longrightarrow}~ 0,
\end{equation*}
as a consequence of \eqref{eq:Characteristic_Def} and of sub-linearity hypothesis \textbn{(C1)}. The same uniform estimate holds true for $\Beta_{\nu}$-almost every $(y,\sigma_{\nu}) \in \R^d \times \Sigma_T$, so that 
\begin{equation*}
\hat{\Beta}_{\mu,\nu} \bigg( \bigg\{ (x,\sigma_{\mu},y,\sigma_{\nu}) ~\text{s.t.}~ \Big| \big(e_t - e_{t_{k_m}},e_t - e_{t_{k_m}} \big)(x,\sigma_{\mu},y,\sigma_{\nu}) \Big| > \epsilon  \bigg\}\bigg) ~\underset{m \rightarrow +\infty}{\longrightarrow}~ 0,
\end{equation*}
for any $\epsilon > 0$. Thus, the sequence of maps $((e_{t_{k_m}},e_{t_{k_m}})) \subset C^0((\R^d \times \Sigma_T)^2,\R^{2d})$ converges in $\hat{\Beta}_{\mu,\nu}$-measure towards $(e_t,e_t) \in C^0((\R^d \times \Sigma_T)^2,\R^{2d})$ as $m \rightarrow +\infty$. Therefore, from classical convergence results on images of measures by sequences of maps (see e.g. \cite[Lemma 5.4.1]{AGS}), we deduce that 
\begin{equation*}
(e_{t_{k_m}},e_{t_{k_m}})_{\#} \hat{\Beta}_{\mu,\nu} ~\underset{m \rightarrow +\infty}{\rightharpoonup^*}~ (e_t,e_t)_{\#} \hat{\Beta}_{\mu,\nu}. 
\end{equation*}
Furthermore, we can again verify that the integrals
\begin{equation*}
\INTDom{|x-y|^p}{\R^{2d}}{\Big( (e_{t_{k_m}},e_{t_{k_m}})_{\#} \hat{\Beta}_{\mu,\nu} \Big)(x,y)} = \INTDom{|\sigma_{\mu}(t_{k_m}) - \sigma_{\nu}(t_{k_m})|^p}{(\R^d \times \Sigma_T)^2}{\hat{\Beta}_{\mu,\nu}(x,\sigma_{\mu},y,\sigma_{\nu})}, 
\end{equation*}
are bounded, uniformly with respect to $m \geq 1$. We can thus invoke the stability under narrow convergence of the sets of $p$-optimal transport plans to recover that 
\begin{equation*}
(e_t,e_t)_{\#} \hat{\Beta}_{\mu,\nu} \in \Gamma_o(\mu(t),\nu(t)), 
\end{equation*}
for all times $t \in [0,T]$. By repeating exactly the same arguments, one can show that 
\begin{equation*}
(\pi_{\R^d},\pi_{\R^d})_{\#} \hat{\Beta}_{\mu,\nu} \in \Gamma_o(\mu^0,\nu^0),
\end{equation*}
which ends our proof of Lemma \ref{lem:OptimalMeasures}. 
\end{proof}

\begin{proof}[Proof (of Proposition \ref{prop:MomentumGronwall})] \textbf{Step 1: Proof of \eqref{eq:MomentumEstimate}.} By the superposition principle stated in Theorem \ref{thm:Superposition}, there exists a measure $\Beta_{\mu} \in \Pcal(\R^d \times \Sigma_T)$ such that $\mu(t) = (e_t)_{\#} \Beta_{\mu}$ for all times $t \in [0,T]$, where $\Beta_{\mu}$ is concentrated on the pairs $(x,\sigma_{\mu}) \in \R^d \times \AC([0,T],\R^d)$ solution of the characteristic equation \eqref{eq:Characteristic_Def}. Therefore for any $p \in [1,+\infty)$, we have 
\begin{equation}
\label{eq:MomentumCalculus1}
\begin{aligned}
\M_p^p(\mu(t)) & =\INTDom{|x|^p}{\R^d}{\mu(t)(x)} \\ 
& = \INTDom{|\sigma_{\mu}(t)|^p}{\R^d \times \Sigma_T}{\Beta_{\mu}(x,\sigma_{\mu})} \\
& \leq \INTDom{ \Big( |x| + \INTSeg{m(s) \big( 1 + |\sigma_{\mu}(s)| \, \big)}{s}{0}{t} \Big)^p}{\R^d \times \Sigma_T}{\Beta_{\mu}(x,\sigma_{\mu})} \\
& \leq 2^{p-1} \Bigg( \INTDom{\Big( |x| + \INTSeg{m(s)}{s}{0}{t} \Big)^p}{\R^d}{\mu^0(x)} + \INTDom{ \Big( \INTSeg{m(s) |\sigma_{\mu}(s)|}{s}{0}{t} \Big)^p}{\R^d \times \Sigma_T}{\Beta_{\mu}(x,\sigma_{\mu})} \Bigg), 
\end{aligned}
\end{equation}
for all $t \in [0,T]$, as a consequence of the sub-linearity hypothesis \textbn{(C1)}.  

Let us denote by $q \in (1,+\infty]$ the conjugate exponent of $p$. Since $\supp(\mu(t)) \subseteq K$ for all times $t \in [0,T]$, we have in particular that $\sigma_{\mu}(\cdot) \in L^{\infty}([0,T],\R^d)$ for $\Beta_{\mu}$-almost every $(x,\sigma_{\mu}) \in \R^d \times \Sigma_T$. This together with the fact that $m(\cdot) \in L^1([0,T],\R_+)$ easily yields 
\begin{equation*}
m(\cdot)^{1/q} \in L^q([0,T],\R_+) \qquad \text{and} \qquad m(\cdot)^{1/p} |\sigma_{\mu}(\cdot)| \in L^p([0,T],\R_+). 
\end{equation*}
We can therefore apply H\"older's inequality to obtain the estimate
\begin{equation}
\label{eq:Holder}
\left( \INTSeg{m(s)|\sigma_{\mu}(s)|}{s}{0}{t} \right)^p \leq  \NormL{m(\cdot)}{1}{[0,t]}^{p/q} \INTSeg{m(s)|\sigma_{\mu}(s)|^p}{s}{0}{t},
\end{equation}
for $\Beta_{\mu}$-almost every $(x,\sigma_{\mu}) \in \R^d \times \Sigma_T$ and $\Lcal^1$-almost every $t \in [0,T]$. Plugging \eqref{eq:Holder} into \eqref{eq:MomentumCalculus1} and applying Fubini's Theorem, we recover 
\begin{equation}
\label{eq:MomentumCalculus2}
\M_p^p(\mu(t)) \leq 2^{p-1} \Bigg( \INTDom{ \Big( |x| + \INTSeg{m(s)}{s}{0}{t} \Big)^p}{\R^d}{\mu^0(x)} + \NormL{m(\cdot)}{1}{[0,t]}^{p/q} \INTSeg{m(s) \M_p^p(\mu(t))}{s}{0}{t} \Bigg). 
\end{equation}
By applying Gr\"onwall's Lemma to \eqref{eq:MomentumCalculus2}, we further obtain
\begin{equation*}
\M_p^p(\mu(t)) \leq 2^{p-1} \Bigg(  \INTDom{ \Big( |x| + \INTSeg{m(s)}{s}{0}{t} \Big)^p}{\R^d}{\mu^0(x)} \Bigg) \exp \left( 2^{p-1} \NormL{m(\cdot)}{1}{[0,t]}^{p/q} \INTSeg{m(s)}{s}{0}{t} \right). 
\end{equation*}
Raising this inequality to the power $1/p$ and using the triangle inequality for the $L^p(\R^d,\R;\mu^0)$-norm, we obtain 
\begin{equation*}
\M_p(\mu(t)) \leq C_p \left( \M_p(\mu(0)) + \INTSeg{m(s)}{s}{0}{t} \right) \exp \left( C_p' \NormL{m(\cdot)}{1}{[0,t]}^p \right)
\end{equation*}
where $C_p = 2^{(p-1)/p}$ and $C_p' = \tfrac{2^{p-1}}{p}$. 

\bigskip

\hspace{-0.6cm}\textbf{Step 2: Proof of \eqref{eq:WassEstimate}.} Again as a consequence of Theorem \ref{thm:Superposition}, there exist two probability measures $\Beta_{\mu},\Beta_{\nu} \in \Pcal(\R^d \times \Sigma_T)$ concentrated on the characteristic curves \eqref{eq:Characteristic_Def} of $v(\cdot,\cdot)$ and $w(\cdot,\cdot)$, such that 
\begin{equation*}
\mu(t) = (e_t)_{\#} \Beta_{\mu} \qquad \text{and} \qquad \nu(t) = (e_t)_{\#} \Beta_{\nu}, 
\end{equation*}
for all times $t \in [0,T]$. For any $p \in [1,+\infty)$, we can invoke Lemma \ref{lem:OptimalMeasures} to build a transport plan $\hat{\Beta}_{\mu,\nu} \in \Pcal((\R^d \times \Sigma_T)^2)$ satisfying
\begin{equation}
\label{eq:OptimalPlan}
(\pi_{\R^d},\pi_{\R^d})_{\#} \hat{\Beta}_{\mu,\nu} := \Bmu_0 \in \Gamma_o(\mu^0,\nu^0) \qquad \text{and} \qquad (e_t,e_t)_{\#} \hat{\Beta}_{\mu,\nu} := \Bmu_t \in \Gamma_o(\mu(t),\nu(t)),
\end{equation}
for all times $t \in [0,T]$. It then holds 
\begin{equation}
\label{eq:EstProof0}
\begin{aligned}
W_p^p(\mu(t),\nu(t)) & = \INTDom{|x-y|^p}{\R^{2d}}{\Bmu_t(x,y)} \\
& = \INTDom{|\sigma_{\mu}(t) - \sigma_{\nu}(t)|^p}{(\R^d \times \Sigma_T)^2}{\hat{\Beta}_{\mu,\nu}(x,\sigma_{\mu},y,\sigma_{\nu})} \\
& \leq \INTDom{\left( |x-y| + \INTSeg{|v(s,\sigma_{\mu}(s)) - w(s,\sigma_{\nu}(s))|}{s}{0}{t} \right)^p}{(\R^d \times \Sigma_T)^2}{\hat{\Beta}_{\mu,\nu}(x,\sigma_{\mu},y,\sigma_{\nu})}. 
\end{aligned}
\end{equation}
We can further estimate the time-integral in the right-hand side of \eqref{eq:EstProof0} as
\begin{equation}
\label{eq:EstProof1}
\begin{aligned}
 \INTSeg{|v(s,\sigma_{\mu}(s)) - w(s,\sigma_{\nu}(s))|}{s}{0}{t} & \leq  \INTSeg{|v(s,\sigma_{\mu}(s)) - v(s,\sigma_{\nu}(s))|}{s}{0}{t} + \INTSeg{|v(s,\sigma_{\nu}(s)) - w(s,\sigma_{\nu}(s))|}{s}{0}{t} \\
& \leq \INTSeg{l_K(s)|\sigma_{\mu}(s) - \sigma_{\nu}(s)|}{s}{0}{t} + \INTSeg{ \NormC{ v(s,\cdot) - w(s,\cdot)}{0}{K,\R^d}}{s}{0}{t}, 
\end{aligned}
\end{equation}
for $\hat{\Beta}_{\mu,\nu}$-almost every $(x,\sigma_{\mu},y,\sigma_{\nu}) \in (\R^d \times \Sigma_T)^2$ as a consequence of hypothesis \textbn{(C2)}. Plugging \eqref{eq:EstProof1} into \eqref{eq:EstProof0}, we recover 
\begin{equation}
\label{eq:EstProof2}
\begin{aligned} 
W_p^p(\mu(t),\nu(t)) \leq 2^{p-1} \INTDom{\left( |x-y| + \INTSeg{\NormC{ v(s,\cdot) - w(s,\cdot)}{0}{K,\R^d}}{s}{0}{t} \right)^p}{(\R^d \times \Sigma_T)^2}{\hat{\Beta}_{\mu,\nu}(x,\sigma_{\mu},y,\sigma_{\nu})} & \\
\hspace{0.45cm} + ~ 2^{p-1} \INTDom{\left( \INTSeg{l_K(s)|\sigma_{\mu}(s) - \sigma_{\nu}(s)|}{s}{0}{t}\right)^p}{(\R^d \times \Sigma_T)^2}{\hat{\Beta}_{\mu,\nu}(x,\sigma_{\mu},y,\sigma_{\nu})}&. 
\end{aligned}
\end{equation}

As in Step 1, we can estimate for $\hat{\Beta}_{\mu,\nu}$-almost every $(x,\sigma_{\mu},y,\sigma_{\nu}) \in (\R^d \times \Sigma_T)^2$ the time integral of the second term in the right-hand side of \eqref{eq:EstProof2} as
\begin{equation*}
\left( \INTSeg{l_K(s)|\sigma_{\mu}(s) - \sigma_{\nu}(s)|}{s}{0}{t} \right)^p \leq \NormL{l_K(\cdot)}{1}{[0,t]}^{p/q} \INTSeg{l_K(s) |\sigma_{\mu}(s) - \sigma_{\nu}(s)|^p}{s}{0}{t}, 
\end{equation*}
where we used H\"older's inequality. This together with Fubini's theorem further yields 
\begin{equation}
\label{eq:EstProof3}
\begin{aligned}
& \INTDom{\left( \INTSeg{l_K(s)|\sigma_{\mu}(s) - \sigma_{\nu}(s)|}{s}{0}{t} \right)^p}{(\R^d \times \Sigma_T)^2}{\hat{\Beta}_{\mu,\nu}(x,\sigma_{\mu},y,\sigma_{\nu})} \\
\leq ~ & \NormL{l_K(\cdot)}{1}{[0,t]}^{p/q} \INTSeg{l_K(s) \INTDom{|\sigma_{\mu}(s) - \sigma_{\nu}(s)|^p}{(\R^d \times \Sigma_T)^2}{\hat{\Beta}_{\mu,\nu}(x,\sigma_{\mu},y,\sigma_{\nu})}}{s}{0}{t} \\
= ~ & \NormL{l_K(\cdot)}{1}{[0,t]}^{p/q} \INTSeg{l_K(s) \INTDom{|x-y|^p}{\R^{2d}}{\Bmu_s(x,y)}}{s}{0}{t} \\
= ~ & \NormL{l_K(\cdot)}{1}{[0,t]}^{p/q} \INTSeg{l_K(s) W_p^p(\mu(s),\nu(s))}{s}{0}{t}. 
\end{aligned}
\end{equation}
Merging \eqref{eq:EstProof2}-\eqref{eq:EstProof3}, we obtain
\begin{equation}
\label{eq:EstProof4}
\begin{aligned} 
W_p^p(\mu(t),\nu(t)) & \leq 2^{p-1} \INTDom{\left( |x-y| + \INTSeg{\NormC{ v(s,\cdot) - w(s,\cdot)}{0}{K,\R^d}}{s}{0}{t} \right)^p}{(\R^d \times \Sigma_T)^2}{\hat{\Beta}_{\mu,\nu}(x,\sigma_{\mu},y,\sigma_{\nu})} \\
& \hspace{0.25cm} + 2^{p-1} \NormL{l_K(\cdot)}{1}{[0,t]}^{p/q} \INTSeg{l_K(s) W_p^p(\mu(s),\nu(s))}{s}{0}{t}. 
\end{aligned}
\end{equation}
As before, applying Gr\"onwall Lemma to \eqref{eq:EstProof4}, raising the resulting inequality to the power $1/p$ and applying the triangle inequality for the $L^p(\R^d,\R;\mu^0)$-norm, we finally recover 
\begin{equation*}
\begin{aligned}
W_p(\mu(t),\nu(t)) \leq \, C_p \left( W_p(\mu(0),\nu(0)) + \INTSeg{\NormC{v(s,\cdot) - w(s,\cdot)}{0}{K,\R^d}}{s}{0}{t} \right) \exp \left( C_p' \NormL{l_K(\cdot)}{1}{[0,t]}^p \right) &, 
\end{aligned}
\end{equation*}
where the constants $C_p, C_p'$ are as in \eqref{eq:ConstExpression}, which concludes the proof of \eqref{eq:WassEstimate}. 
\end{proof} 

\smallskip

\begin{flushleft}
{\small{\bf  Acknowledgement.}  This material is based upon work supported by the Air Force Office of Scientific Research under award number FA9550-18-1-0254.}

{\small The authors are also grateful to the referee for constructive comments that helped improving the earlier version of this manuscript.}
\end{flushleft}


\bibliographystyle{plain}
{\footnotesize
\bibliography{../../Articles/ControlWassersteinBib}}

\end{document}